\documentclass[11pt,a4paper,twoside]{amsart}
\usepackage{amssymb,amsmath,amsthm,graphicx,color
}
\theoremstyle{plain}
\newtheorem{theorem}{Theorem}[section]

\newtheorem{lemma}[theorem]{Lemma}

\newtheorem{proposition}[theorem]{Proposition}
\newtheorem{assumption}[theorem]{Assumption}
\theoremstyle{remark}
\newtheorem{remark}[theorem]{Remark}

\allowdisplaybreaks[4]

\usepackage{amsrefs}

\numberwithin{equation}{section}

\newcommand{\C}{\mathbb{C}}
\newcommand{\R}{\mathbb{R}}
\newcommand{\Z}{\mathbb{Z}}

\newcommand{\F}{\mathcal{F}}

\renewcommand{\Im}{\operatorname{Im}}
\renewcommand{\Re}{\operatorname{Re}}
\newcommand{\I}{\infty}
\newcommand{\abs}[1]{\left\lvert #1\right\rvert}
\newcommand{\wha}[1]{\widehat{#1}}

\newcommand{\norm}[1]{\left\lVert #1\right\rVert}
\newcommand{\Lebn}[2]{\left\lVert #1 \right\rVert_{L^{#2}}}
\newcommand{\Sobn}[2]{\left\lVert #1 \right\rVert_{H^{#2}}}

\newcommand{\Jbr}[1]{\left\langle #1 \right\rangle}
\newcommand{\tnorm}[1]{\lVert #1\rVert}

\newcommand{\m}[1]{\mathcal{#1}}

\def\Sch{{\mathcal S}} 
\def\({\left(}
\def\){\right)}
\def\<{\left\langle}
\def\>{\right\rangle}
\def\le{\leqslant}
\def\ge{\geqslant}

\def\d{{\partial}}

\def \wt{\widetilde}
\def \wh{\widehat}

\def \f{\phi}
\def \e{\varepsilon}
\def \l{\lambda}

\def \d{\delta}
\def \D{\Delta}

\def \pa{\partial}
\def \n{\nabla}

\def \a{\alpha}
\def \b{\beta}

\def \n{\nabla}
\def \t{\theta}

\def \P{\Phi}

\def \F{\mathcal{F}}
\def \mL{\mathcal{L}}
\def \mN{\mathcal{N}}
\def \mG{\mathcal{G}}
\def \mR{\mathcal{R}}

\newcommand{\eps}{\varepsilon}

\DeclareMathOperator{\Lip}{Lip}

\newcommand{\todayd}{\the\year/\the\month/\the\day}

\theoremstyle{definition}
\begin{document}
\title[Long range scattering for NLS]
{Long range scattering for nonlinear Schr\"odinger equations with critical homogeneous nonlinearity in three space dimensions}

\author[S. Masaki]{Satoshi MASAKI}
\address[]{Division of Mathematical Science, Department of Systems Innovation, Graduate School of Engineering Science, Osaka University, Toyonaka, Osaka, 560-8531, Japan}
\email{masaki@sigmath.es.osaka-u.ac.jp}
\author[H. Miyazaki]{Hayato MIYAZAKI}
\address[]{Advanced Science Course, Department of Integrated Science and Technology, National Institute of Technology, Tsuyama College, Tsuyama, Okayama, 708-8509, Japan}
\email{miyazaki@tsuyama.kosen-ac.jp}
\author[K. Uriya]{Kota URIYA}
\address[]{Department of Applied Mathematics, Faculty of Science, Okayama University of Science, Okayama, Okayama, 700-0005, Japan}
\email{uriya@xmath.ous.ac.jp}

\date{}

\maketitle
\vskip-5mm
\begin{abstract}
In this paper, we consider the final state problem for the nonlinear Schr\"odinger equation with a homogeneous nonlinearity of the critical order which is not necessarily a polynomial.
% in three space dimensions. 
%Because the nonlinearity has the critical order, asymptotic behavior of a solution is determined by the shape of the nonlinearity. 
In \cite{MM2}, the first and the second authors consider one- and two-dimensional cases and
gave a sufficient condition on the nonlinearity for that the corresponding equation admits
a solution that behaves like a free solution with or without a logarithmic phase correction.
The present paper is devoted to the study of the three-dimensional case, 
in which it is required that a solution converges to a given asymptotic profile in a faster rate than in
the lower dimensional cases. 
To obtain the necessary convergence rate, %a new ingredient is to 
we employ the end-point Strichartz estimate and modify a time-dependent regularizing operator, introduced in \cite{MM2}.
%Then lack of differentiability of the nonlinearity disturbs getting necessary decay order.
%To overcome difficulty, we exploit the end-point Strichatrz estimate and improve the time-dependent operator by \cite{MM2}.
Moreover, we present a candidate of the second asymptotic profile to the solution. 
\end{abstract} 
\section{Introduction}

In this paper, we consider large time behavior of solutions to
 nonlinear Schr\"odinger equation
\begin{equation}\label{eq:NLS}\tag{NLS}
	i \pa_t u + \Delta u = F(u). %\quad u(0,x) = u_0(x)
\end{equation}
Here, $(t,x) \in \R^{1+d}$ and $u=u(t,x)$ is a complex-valued unknown function. 
We suppose that the nonlinearity $F$ is homogeneous of degree $1+2/d$, that is, $F$ satisfies
\begin{equation}\label{eq:cond1}
	F(\lambda u) = \lambda^{1+\frac2d} F(u)
\end{equation}
for any $u\in \C$ and $\lambda >0$. 
This is the continuation of the previous study in \cite{MM2}.
In \cite{MM2}, we consider one- and two-dimensional cases and
give a sufficient condition on $F:\C \to \C$ for existence of a modified wave operator, that is, for
that \eqref{eq:NLS} admits a nontrivial solution which asymptotically behaves like
\begin{equation}\label{eq:masymptotic}
	u_p (t) = (2it)^{-\frac{d}{2}}  e^{i \frac{|x|^2}{4t}} \wha{u_{+}}\(\frac{x}{2t}\) \exp \( -i \mu \left| \wha{u_{+}}\(\frac{x}{2t}\) \right|^{\frac{2}{d}} \log t \)
\end{equation}
as $t\to\I$, where $u_+$ is a given final data and $\mu$ is a real constant determined by $F$.
We would remark that it is applicable to non-polynomial nonlinearities such as $|\Re u|\Re u$. 
The aim here is to extend the previous result to the case $d=3$.
Because the exponent $1 + 2/d$ becomes small in high dimensions, we face
some difficulties such as lack of differentiability of the nonlinearity.
As for the nonlinearity $F(u) = \l |u|^{2/3}u$,
Ginibre-Ozawa~\cite{GO} showed that a class of solutions
has the asymptotic profile \eqref{eq:masymptotic} with $\mu=\l$. 
However, it seems that no other homogeneous nonlinearity is treated so far.

\medskip
In \cite{MM2}, a sufficient condition on the nonlinearity $F$ for existence of a modified wave operator
is given in terms of the ``Fourier coefficients'' of the nonlinearity. 
The crucial step of construction of a modified wave operator
is to find an asymptotic behavior that actually takes place.
For this part, specifying a \emph{resonant part} of the nonlinearity, which 
determines the shape of the asymptotic behavior, is essential.
A new ingredient in \cite{MM2} is the expansion of the nonlinearity into an \emph{infinite sum} via Fourier series expansion.
For example, the nonlinearity $F(u)= |\Re u | \Re u$ is written as
\begin{equation}\label{eq:example1}
	|\Re u| \Re u = \frac{4}{3\pi} |u| u + \sum_{m\neq 0} \frac{4(-1)^{m+1}}{\pi (2m-1)(2m+1)(2m+3)} |u|^{1-2m}u^{1+2m}.  
\end{equation}
% The expansion well extracts the resonant part.
The first gauge-invariant term $\frac{4}{3\pi} |u| u$ is the resonant part and 
the remaining infinite sum is a \emph{non-resonant part}.
It turns out that the possible asymptotic behavior of solutions to \eqref{eq:NLS} with $d=2$ is \eqref{eq:masymptotic} with $\mu=4/3\pi$.

Once we find a ``right'' asymptotic behavior,
it is possible to construct a solution around the asymptotic profile.
For this, we shall show that the non-resonant part is negligible for large time.
Note that the non-resonant part is a sum of ``gauge-variant'' nonlinearities.
Because of the gauge-variant property, the non-resonant part has different phase from the solution itself.
The disagreement causes an extra time decay effect (cf.~stationary phase) and so the effect of the non-resonant term
becomes relatively small for large time.
The case where the non-resonant part is a finite sum is previously treated in \cite{ShT,HNST,HWN}.
The main technical issue to treat general nonlinearity lies in 
showing that the non-resonant part which consists of infinitely many term is still acceptable (see \cite{MM2}).

\medskip
In this paper, we will extend the technique to the three-dimensional case.
The argument in \cite{MM2} is not directly applicable.
To construct a solution around a given asymptotic profile in three dimensions,
it is required that the solution converges to 
the asymptotic profile faster than in the one- and two-dimensional cases. 
Since the rate is controlled by the 
time decay rate of the non-resonant part, we need a good decay property of the non-resonant part.
However, lack of differentiability in three-dimensions then disturbs obtaining such fast decay property. 

To overcome this difficulty, we modify the 
argument of \cite{MM2} in two respects. 
The first one is that 
we enlarge the function space to construct a solution by
employing the end-point Strichartz estimate. 
This enable us to reduce the necessary condition on the rate of convergence
of the solution. We notice that the end-point Strichartz 
estimate is peculiar to the space of dimensions other than two 
(see \cite{KT}). 
The second respect is to improve the estimate for the high 
frequency part of the non-resonant part, which yields a 
better decay rate of the solution. 
However, we still assume that the given final data has very small low-frequency
part. We remark that if a final data has a non-negligible low-frequency part then 
there appear other kinds of asymptotic behavior (see 
\cite{HN02, HN04, HN11, HN15, N, NS}).
%\cite{HN02}, 
%\cite{HN04}, \cite{HN11}, \cite{HN15}, \cite{N}, \cite{NS}
\medskip

In order to present the main result, let us briefly recall the decomposition of the nonlinearity in \cite{MM2}.
We identify a homogeneous nonlinearity $F$ and $2\pi$-periodic function $g$ as follows.
A homogeneous nonlinearity $F$ is written as
\begin{equation}\label{eq:id1}
	F(u) = |u|^{\frac53} F\( \frac{u}{|u|} \).
\end{equation}
We then introduce a $2\pi$-periodic function $g(\theta)=g_F(\theta)$ by $g_F(\theta) = F(e^{i\theta})$.
Conversely, for a given $2\pi$-periodic function $g$, we can construct a homogeneous nonlinearity $F=F_g:\C \to \C$ by
$F_g(u) = |u|^{\frac53} g\( \arg u \)$ if $u\neq 0$ and
$F_g(u) = 0$ if $u=0$.
Since $g(\theta)$ is $2\pi$-periodic function, it holds, at least formally, that $g(\theta)=\sum_{n\in\Z} {g}_n e^{in\theta}$, where
\begin{equation}\label{eq:gn}
	{g}_n := \frac{1}{2\pi} \int_0^{2\pi} g(\t) e^{-in\t}d\t.
\end{equation}
Remark that the expansion gives us
\begin{equation}\label{eq:exp}
	F(u) = g_0 |u|^{\frac53} + {g}_1 |u|^{\frac23}u + \sum_{n\neq 0,1} {g}_n |u|^{\frac53-n} u^n.
\end{equation}
\subsection{Main results}

Set $\Jbr{a}=(1+|a|^2)^{1/2}$ for $a \in \C$ or $a\in \R^3$. Let $s$, $m \in \R$. The weighted Sobolev space on $\R^{3}$ is defined by $H^{m,s} = \{u \in \Sch'\;  ;\; \Jbr{i\n}^m \Jbr{x}^s u \in L^2 \}$. 
Let us simply write $H^{m} = H^{m,0}$.
We denote by $\norm{g}_{\mathrm{Lip}}$ the Lipschitz norm of $g$.
% We say a pair $(q,r) \in [2,\I]^2$ is admissible if $\frac2q+\frac3r=\frac32$.

Throughout the paper, we suppose the following:
\begin{assumption}\label{asmp:main}
Assume that $F$ is a homogeneous nonlinearity of degree $5/3$
 such that a corresponding $2\pi$-periodic function
$g(\theta)$ satisfies $g_0=0$, $g_1 \in \R$,
and
\[
\sum_{n \in \Z}|n|^{1+\eta} |g_n| < \I
\]
for some $\eta>0$, where $g_n$ is given in \eqref{eq:gn}.
In particular, $g$ is Lipschitz continuous.
\end{assumption}

\begin{theorem}[Existence and uniqueness]\label{thm:main}
Suppose that the nonlinearity $F$ satisfies Assumption \ref{asmp:main} for $\eta>0$.
Fix $\d \in (3/2,5/3)$ so that $\delta-3/2<2\eta$.
% Take $ b \in(3/4, \d/2)$.
Then, there exists $\eps_0=\eps_0(\norm{g}_{\mathrm{Lip}})$ such that 
for any $u_+ \in H^{0,2} \cap {H}^{-\d}$ satisfying
$\Lebn{\wha{u_+}}{\I}<\eps_0$ there exists $T>0$ and a solution $u \in C([T,\I); L^2(\R^3))$ of \eqref{eq:NLS} 
which satisfies
\begin{equation}\label{eq:mainest}
\sup_{t \in [T, \I)}t^{b} \Lebn{u(t) -u_p(t)}{2}  < \I
\end{equation}
for any $b<\delta/2$,
where
\begin{equation}\label{eq:masymptotic2}
u_p (t) := (2it)^{-\frac{3}{2}}  e^{i \frac{|x|^2}{4t}} \wha{u_{+}}\(\frac{x}{2t}\) \exp \( -i g_1 \left| \wha{u_{+}}\(\frac{x}{2t}\) \right|^{\frac{2}{3}} \log t \).
\end{equation}
The solution is unique in the following sense:
If $\tilde{u} \in C([\tilde{T},\I);L^2(\R^3))$ solves \eqref{eq:NLS} and satisfies \eqref{eq:mainest}
for some $\tilde{T}$ and $b>3/4$ then $\tilde{u}=u$.
\end{theorem}

The following theorem describes the asymptotic behavior more precisely.

\begin{theorem}[Asymptotic behavior]\label{thm:main2}
Under the same assumption as in Theorem \ref{thm:main}, the solution $u(t)$ given in Theorem \ref{thm:main}
satisfies
\begin{equation}\label{eq:mainest2}
\sup_{t \in [T,\I)} t^b \norm{u-u_p-\mathcal{V}}_{L^\I_t([t,\I);L^2_x) \cap L^2_t([t,\I);L^6_x)} <\I
\end{equation}
for any $b<\d/2$, where
\begin{equation}\label{def:calV}
% 	\mathcal{V}(t)
% 	:= i \F^{-1} \sum_{n\neq0,1} g_n  D\(\frac{n}{2}\) \frac{|\wha{u_+}|^{\frac53-n}
% 	\wha{u_+}^n}{1+in(n-1)t|x|^2} 
% 	\exp(-in(t|x|^2+ g_1 |\wha{u_{+}}|^{\frac{2}{3}} \log t)).
	\mathcal{V}(t)
	:= - \F^{-1} \sum_{n\neq0,1} \frac{g_n}{2(in)^{3/2}}  \(\frac{|\wha{u_+}|^{\frac53-n}
	\wha{u_+}^n i^{-\frac32 n} e^{-int|\cdot|^2} e^{-in g_1 |\wha{u_{+}}|^{\frac{2}{3}} \log t}}{1+in(n-1)t|\cdot|^2} 
	\)\(\frac{\xi}n\).
\end{equation}
In the $L^\I([T,\I);L^2)$-topology, $\mathcal{V}$ is small: For any $b<\d/2$,
\begin{equation}\label{eq:behavior1}
\sup_{t \in [T,\I)} t^b \norm{\mathcal{V}}_{L^\I_t([t,\I);L^2_x)} <\I.
\end{equation}
In the $L^2([T,\I);L^6)$-topology, it holds that
\begin{equation}\label{eq:behavior2}
\sup_{t \in [T,\I)} t^b \norm{\mathcal{V}-v_p}_{L^2_t([t,\I);L^6_x)} <\I
\end{equation}
for any $b<\d/2$, where
\begin{equation}\label{def:vp}
	v_p(t)
% 	= i \sum_{n\neq0,1} g_n \(1+i \frac{n-1}n t\Delta \)^{-1} t
% 	|u_p|^{\frac53-n}
% 	u_p^n 
	:= -i \sum_{n\neq0,1} g_n \frac1{t^{-1}-i \frac{n-1}n \Delta }
 	|u_p(t)|^{\frac53-n}
 	u_p(t)^n.
\end{equation}
\end{theorem}
\begin{remark}
A straightforward estimate shows 
$\norm{v_p}_{L^2_t([T,\I);L^6_x)} \le C T^{-\frac12}$.
However, we do not have a lower bound of $v_p$ so far.
If this estimate is sharp then $v_p$ and $\mathcal{V}$ are
true second asymptotic profiles of the solution in $L^2_tL^6_x$-topology.
On the other hand, if $v_p$ (and so $\mathcal{V}$) is small also in $L^2_tL^6_x$-topology,
then the asymptotics \eqref{eq:mainest2} holds with $\mathcal{V}=0$, which means
the asymptotic behavior of $u_p$ is the same as that in the case $F(u)=\l |u|^{2/3}u$.
\end{remark}
\begin{remark}
In the case $F(u)=\l |u|^{2/3}u$,
our estimate \eqref{eq:mainest2} is an improvement of that in \cite{GO}.
More precisely, it improves possible range of $b$ and includes the endpoint case $L^2_tL^6_x$.
\end{remark}

\begin{remark}\label{rmk:MTT}
% If $|x|^{-2-\zeta} |\wha{u_+}|^{5/3} \in L^2$ for some $\zeta>0$ then we have
% $\mathcal{V}(t) = - \frac{2t}{|x|^2} \tilde{F}(u_p(t)) + O(t^{-1-\frac\zeta2})$
% in $L^2$ as $t\to\I$, 
Under suitable additional assumptions such as $u_+ \in \dot{H}^{-2-}$, we have
$\mathcal{V}(t) = \tilde{F}(|\frac{2t}{x}|^{6/5} u_p(t)) + o(t^{-1})$ in $L^2$,
where $\tilde{F}(u)$ is a homogeneous nonlinearity 
such that the corresponding Fourier coefficients are $\tilde{g}_n=\frac{1}{n(1-n)} g_n$.
The asymptotic profile $\tilde{F}(|\frac{2t}{x}|^{6/5}u_p(t))$ is a natural extension of those 
used in \cite{MTT,ShT}.
\end{remark}
\begin{remark}\label{rmk:example}
Our theorem can be applied to
$F(u) = |\Re u|^{\frac23} \Re u$.
The corresponding periodic function is
$g(\theta) = |\cos \t|^{\frac23} \cos \t$
and so 
\[
g_n = \left\{ 
\begin{aligned}
&  \frac{(-1)^{\frac{n-1}2} \Gamma(\frac{11}6) \Gamma(\frac{3n-5}6)}{\sqrt\pi 
	\Gamma(- \frac{1}3) \Gamma(\frac{3n+11}6)}  && n\text{: odd},\\
& 0 && n \text{: even}.
\end{aligned}
\right.
\]
In particular, $g_n=O(|n|^{-8/3})$ as $|n|\to\I$.
See Appendix \ref{sec:appendix1} for the details.
% We also remark that the corresponding $u_p$ is the same as for the nonlinearity $F(u)= \frac{\Gamma (11/6)}{\sqrt{\pi}\Gamma (7/3)} |u|^{\frac23}u$.
\end{remark}

\begin{remark}
Theorem \ref{thm:main} implies that when $F$ satisfies Assumption \ref{asmp:main} and $g_1=0$, \eqref{eq:NLS} admits a nontrivial solution which has the asymptotic profile
\begin{equation}\label{eq:fasymptotic}
	u_p(t)= (2it)^{-\frac{d}{2}}  e^{i \frac{|x|^2}{4t}} \wha{u_{+}}\(\frac{x}{2t}\) .
\end{equation}
Notice that this is nothing but the asymptotic behavior of the linear solution $e^{it\Delta} u_+$,
and so that our theorem implies that the equation admits an asymptotic free solution in this case.
The nonlinearity $F(u)=|\Re u|^\frac23 \Re u - i|\Im u|^\frac23 \Im u$ is such an example
(See Appendix \ref{sec:appendix1}).
\end{remark}

\subsection{Strategy and Improvements}
Let us briefly outline the proof of Theorems \ref{thm:main} and \ref{thm:main2}.
The strategy is the same spirit as in \cite{MM2}.
By the decomposition \eqref{eq:exp} and by
Assumption \ref{asmp:main}, we write
\begin{equation}\label{eq:decomp}
	F(u) = g_1 |u|^{\frac23}u + \sum_{n\neq 0,1} g_n |u|^{\frac53-n} u^n.
\end{equation}
Denote
\[
	\mG (u) := g_1 |u|^{\frac23}u, \quad  \mN(u)  := \sum_{n\neq 0,1} g_n |u|^{\frac53-n} u^n.
\]
$\mG$ corresponds to the resonant part and $\mN$ to the non-resonant part.
We then introduce a formulation in \cite{HWN} (see also \cite{HNST,ShT, HN06}).
% In what follows, we let $t>1$ unless otherwise stated.
Let $U(t)=e^{it\D}$.
Introduce a multiplication operator $M(t)$ and a dilation operator $D(t)$ by
\begin{equation*} %\label{def:MD}
	M(t) = e^{\frac{i |x|^2}{4t}}, \quad (D(t)f)(x) = (2it)^{-\frac{3}{2}} f\(\frac{x}{2t}\).
\end{equation*}
They are isometries on $L^2(\R^3)$. Then, $u_p$ is written as $u_p(t) = M(t)D(t) \widehat{w}(t)$ with
\begin{equation}\label{def:upw}
	\widehat{w}(t) := \wha{u_{+}} \exp(-i g_1 |\wha{u_{+}}|^{\frac{2}{3}} \log t).
\end{equation}
Note that $|\wha{w}(t.x)|=|\wha{u_+}(x)|$.
We regard the equation \eqref{eq:NLS} as
\[
\mL(u-u_{p}) = F (u) -F (u_p) - \mL u_p + \mG (u_{p}) + \mN(u_p),
\]
where $\mL=i\pa_t + \Delta_x$.
A computation shows that it is rewritten as the following integral equation;
\begin{equation}\label{eq:inteq}
	u(t) -u_p(t) = i \int_t^{\I} U(t-s) \( F (u) - F (u_p) \)(s) ds + \mathcal{E}_{\mathrm{r}}(t) + \mathcal{E}_{\mathrm{nr}}(t),
\end{equation}
where external terms are defined by
\begin{align}\label{def:external}
	\mathcal{E}_{\mathrm{r}}(t):={}&\mR(t) \wha{w} -i \int_t^{\I} U(t-s) \mR (s) \mG (\wha{w}) (s) \frac{ds}{s}, \\
	\label{def:external2}
	\mathcal{E}_{\mathrm{nr}}(t):={}&i \int_t^{\I} U(t-s)\mN (u_p)(s) ds,
\end{align}
with
\[
\mR(t) = M(t) D(t)\(U\(-\frac{1}{4t}\) -1\)
\]
(see \cite{HWN} for the details).

For $R>0$, $T\ge1$, and $b>0$, we define a complete metric space
\begin{align*}
X_{T, b, R} &:= \{v \in C([T,\I); L^2(\R^3)) ;\  \norm{v-u_p}_{X_{T, b}} \le R \}, \\
\norm{v}_{X_{T,b}} &:= \sup_{t \in [T, \I)}t^{b} \norm{v(t)}_{L^2(\R^3)} = \sup_{t \in [T, \I)}t^{b} \norm{v}_{L^\I_t ([t,\I); L^2(\R^3))}, \\
d (u,v) &:= \norm{u-v}_{X_{T,b}} .
\end{align*}
It is easy to see that $X_{T_1, b_1, R_1} \subset X_{T_2, b_2, R_2}$ if $(1 \le )T_1 \le T_2$, $b_1\ge b_2$, and $R_1\le R_2$.
When the asymptotic profile $u_p$ is suitably chosen, we can construct a solution in $X_{T,b,R}$
for some $T,b,R$.
The appropriateness can be stated as the existence of $T_0\ge 1$ such that
\begin{equation}\label{eq:main_step_of_proof}
	\norm{\mathcal{E}_{\mathrm{r}} + \mathcal{E}_{\mathrm{nr}}}_{X_{T_0,b}} <\I,
\end{equation}
where $\mathcal{E}_{\mathrm{r}}$ and $\mathcal{E}_{\mathrm{nr}}$ are given in 
\eqref{def:external} and \eqref{def:external2}, respectively.
The solvability of \eqref{eq:inteq} under this assumption will be discussed in Section \ref{sec:abst}.
Then, it will turn out that we need to choose $b>3/4$.
\begin{remark}
The condition for $b$ is $b>d/4$ in dimensions $d=1,2$ (see \cite{MM2}), and so
the above condition is a natural extension.
\end{remark}
\begin{remark}\label{aux1:1}
An improvement lies in the definition of $X_{t,b}$-norm.
In the previous paper \cite{MM2}, the norm has one more term
\begin{equation}\label{eq:auxy_norm}
	\sup_{t \in [T, \I)}t^{b}\norm{v}_{L^{q}_t([t,\I) ,L^r_x(\R^d))},
\end{equation}
where $(q,r)=(4,\I)$ if $d=1$ and $(q,r)=(4,4)$ if $d=2$ are admissible pairs.
In the three-dimensional case, we are able to remove this kind of auxiliary norm by means of the endpoint Strichartz estimate.
Theorem \ref{thm:main2} suggests that the exponent $b$ for which \eqref{eq:auxy_norm} can be bounded
actually depends on the choice of $(q,r)$.
\end{remark}

The main step of the proof of main theorems is the following.
\begin{proposition} \label{prop:main}
Let $3/2 < \d < 5/3$. 
Assume that $\sum_{n\in\Z}|n|^{1+\eta}|g_n| < \I$ for some $\eta> \frac12(\delta-\frac{3}2)$.
For any $u_+ \in H^{0,2} \cap {H}^{-\d}$,
there exists a constant $C=C(g_1,\norm{u_+}_{H^{0,2} \cap {H}^{-\d}})$ such that
\begin{equation}
	\norm{\mathcal{E}_\mathrm{r} + \mathcal{E}_\mathrm{nr}}_{L^{\I}_t ([T, \I) ; L^2)} 
	\le CT^{-\frac{\d}2} \Jbr{\log T}^3 \sum_{n \in \Z} |n|^{1+\eta} |g_n|
\end{equation}
and
\begin{equation}
	\norm{\mathcal{E}_\mathrm{r} + \mathcal{E}_\mathrm{nr} - \mathcal{V}}_{L^{\I}_t ([T, \I) ; L^2)\cap L^2([T,\I);L^6)} 
	\le CT^{-\frac{\d}2} \Jbr{\log T}^3 \sum_{n \in \Z} |n|^{1+\eta} |g_n|
\end{equation}
holds for all $T\ge 2$, where $\mathcal{V}$ is given in \eqref{def:calV}.
\end{proposition}
The first estimate shows that \eqref{eq:main_step_of_proof} holds for $3/4 < b < \d/2$.
We then obtain Theorem \ref{thm:main}.
The second estimate is a main step of the proof of \eqref{eq:mainest2}.
Combining some other estimates on $\mathcal{V}$, we obtain Theorem \ref{thm:main2}.

The main technical part lies in the estimate of $\mathcal{E}_{\mathrm{nr}}$.
We briefly recall previous results to explain how to handle the term. %$\mathcal{E}_{\mathrm{nr}}$. 
In \cite{HNST}, Hayashi, Naumkin, Shimomura, and Tonegawa introduced
an argument to show the time decay of 
the non-resonant part by means of integration by parts.
The decay comes from the fact that the phase of the non-resonant part is different from 
that of the linear part.
Their method however requires higher differentiability of the nonlinearity. 
In order to reduce the required differentiability of the nonlinearity, 
Hayashi, Naumkin, and Wang~\cite{HWN} employ a time-dependent smoothing operator (essentially a cutoff to
the low-frequency part) and
apply the integration by parts only to the low-frequency part. 
In \cite{MM2},  
the frequency cutoff is chosen dependently also on the ``Fourier mode''
to treat an infinite Fourier series expansion of the nonlinearity.

The time decay estimates of the high-frequency part in \cite{HWN, MM2}
are based on the fact that the regularizing operator converges to the identity operator as time goes to infinity. 
So, the only way to improve the estimate would seem to ``lessen'' the high-frequency part by modifying the 
regularizing operator so that it converges in a faster rate.
However, if we do so, %improve the decay of the {high-frequency} part by modifying the cutoff, 
then the estimate for the {low-frequency} part becomes worse.
The loss may not be recovered by refining the estimate on the low-frequency part
because such a refinement requires differentiability more than that nonlinearities satisfying \eqref{eq:cond1} possess. 
% This is the difficulty.

To resolve the difficulty, we improve the estimate for the high-frequency part in another way.
We work with a regularizing operator which has a \emph{flatness property}.
This enable us to use a regularizing operator even \emph{milder} than that used in \cite{HWN,MM2}.
For the details, see Remark \ref{rmk:regu}.
As a result, it reduces the required differentiability of the nonlinearity.
The idea is also applicable to the two-dimensional case and improves 
the previous result in \cite{MM2}.
However, we do not pursue it here.
\medskip

The rest of the paper is organized as follows.
In Section \ref{sec:preliminary}, we summarize useful estimates.
The improve estimate for regularizing operator is discussed here.
Section \ref{sec:abst} is devoted to the proof of main theorems in an abstract form.
Then, it will turn out that our main result is a consequence of Proposition \ref{prop:main}.
Finally, we prove Proposition \ref{prop:main} in Section \ref{sec:mainlem}.

\section{Preliminaries}\label{sec:preliminary}
\subsection{An estimate for regularizing operator}

To obtain time decay property of the non-resonant part $\mathcal{E}_{\mathrm{nr}}$,
we improve an estimate for the high-frequency part.
In this subsection, we consider general space dimensions $d$.
We denote the homogeneous Sobolev space on $\R^d$ by $\dot{H}^{m} =\{ u \in \Sch'\;  ;\; (-\D)^{\frac{m}{2}}u \in L^2\}$. 
Let $\psi \in \mathcal{S}$.
We introduce a regularizing operator $\m{K}_\psi=\m{K}_\psi(t,n)$ by
\begin{equation}\label{def:Kpsi}
	\m{K}_\psi := \psi\( \frac{i\nabla}{|n| \sqrt{t}} \):=\F^{-1} %e^{-\frac{|\xi|^2}{4|n|^{2\b}t^{\s}}}
	\psi \( \frac{\xi}{|n| \sqrt{t}} \) \F.
\end{equation}
We have an equivalent expression
\[
\m{K}_\psi f = C_d ((|n| \sqrt{t} )^{d}  [\F^{-1}\psi]( |n| \sqrt{t} \cdot ) * f)(x).
\]
The following is an improvement of \cite{MM2} by using a kind of isotropic property of $\psi$ near the origin. %is flat at the origin.
\begin{lemma}[Boundedness of $\m{K}$] \label{mol:1.1}
Take $\psi \in \mathcal{S}$ and set $\m{K}_\psi$ as in \eqref{def:Kpsi}.
Let $s\in \R $ and $\theta \in [0,2]$.
Assume $\nabla \psi(0)=0$ if $\theta \in (1,2]$. For any $t>0$ and $n\neq 0$, the followings hold.
\begin{enumerate}
\renewcommand{\labelenumi}{(\roman{enumi})}
\item
$\m{K}_\psi$ is a bounded linear operator on $L^2$ and satisfies $\norm{\m{K}_\psi}_{\m{L}(L^2)} \le \norm{\psi}_{L^\I}$.
Further, $\m{K}_\psi$ commutes with $\nabla$. In particular, $\m{K}_\psi$ is a bounded linear operator on $\dot{H}^s$ and satisfies $\norm{\m{K}_\psi}_{\m{L}(\dot{H}^s)} \le \norm{\psi}_{L^\I}$.
\item 
$\m{K}- {\psi}(0)$ is a bounded linear operator from $\dot{H}^{s+\theta}$ to $\dot{H}^{s}$ with norm
$$\norm{\m{K}_\psi-\psi(0)}_{\m{L}(\dot{H}^{s+\t} ,\dot{H}^{s})} \le Ct^{-\frac{\theta}{2}}|n|^{-{\t}}.$$
\end{enumerate}
\end{lemma}
\begin{proof}
The first item is obvious.
% Let us prove the second. By Young's inequality
% \begin{align*}
% 	\norm{ \m{K}_\psi f }_{L^{r_2}} \le{}& C \norm{(|n| t^{\s/2} )^{d}  \F^{-1}\psi( |n| t^{\s/2} \cdot )}_{L^p}
% 	\norm{f}_{L^{r_1}} \\
% 	={}& C (|n| t^{\s/2} )^{d(1-\frac1p)} \norm{f}_{L^{r_1}} = C (|n| t^{\s/2} )^{d(\frac1{r_1}-\frac1{r_2})}
% \end{align*}
% for $ 1/p= 1/r_2 -  1/r_1 +1 \in [0,1]$.
% We finally prove the third item. 
Let us prove the second. We consider the case $\theta \in (1,2]$ and $\nabla \psi (0) = 0$, the other case is the same as in \cite{HWN}.
It suffices to show the case $s=0$. By assumption $\nabla \psi(0)=0$, we have
\[
	\int_{\R^d} y \F^{-1} \psi (y) dy =0
\]
For $\f \in \dot{H}^{\theta}$, one sees from the equivalent expression that
\begin{align*}
	&{}[(\m{K}_\psi-{\psi}(0))\f](x) \\
	={}& C_d ({|n| \sqrt{t}} )^{d} \int_{\R^d} \F^{-1} \psi(|n| \sqrt{t}y) (\f(x-y) - \f (x) + y\cdot\nabla \f(x)) dy.
\end{align*}
Remark that
\[
	\norm{\f(\cdot- y) - \f + y\cdot\nabla \f}_{L^2_x}
	=\norm{(e^{-iy\cdot\xi}-1+iy\cdot\xi)\F\f}_{L^2_\xi}
	\le C |y|^\theta \norm{\f}_{\dot{H}^\theta}
\]
for $\theta \in [1,2]$. By these estimates,
\begin{align*}
\Lebn{(\m{K}_\psi-{\psi}(0))\f}{2} &\le C_d ({|n| \sqrt{t}} )^{d} \int_{\R^d} |\F^{-1} \psi(|n| \sqrt{t}y)| |y|^\theta \norm{\f}_{\dot{H}^\theta} dy \\
&\le C_\psi t^{-\frac{\theta}{2}}|n|^{-{\t}}\norm{\f}_{\dot{H}^\theta}.
\end{align*}
The proof is completed.
\end{proof}

\begin{remark}\label{rmk:regu}
It is the property $\nabla \psi(0)=0$ that allows us to take $\t \in (1,2]$ in Lemma \ref{mol:1.1} (ii).
The property implies that the corresponding cutoff operator $\mathcal{K}_{\psi}$ is a ``flat'' cutoff.
It was not used in \cite{HWN,MM2} and so $\t$ is restricted to $\theta \le 1$.
If $d\ge2$, the time decay $t^{-1/2}$ for the high-frequency part, which is given with $\theta=1$, is not sufficient.
To recover the lack of decay, the operator of the form
$\psi\( {|n|^{-1} t^{-\sigma/2}} (i\nabla) \)$ was used with $\sigma >1$.
This makes the estimate of the high-frequency part better but that of the low-frequency part worse, in view of  the time decay rate and order in $n$.
In particular, the low-frequency part generated by the operator is considerably large
and so it causes some loss in the integration-by-parts procedure.
\end{remark}

\begin{remark}
It is easy to see that if $\psi \in \mathcal{S}$ satisfies $\psi\equiv1$ in the neighborhood of the origin,
we have no upper bound on $\theta$ in Lemma \ref{mol:1.1}.
\end{remark}

\subsection{Fractional chain rule of homogeneous functions of order $5/3$}

% Due to a lack of differentiability,
% an estimate of $\||\wha{w}|^{5/3-n}\wha{w}^n\|_{H^\d}$ needs some modifications.

Let us collect useful estimates on the estimate of the nonlinearity satisfying \eqref{eq:cond1}.
In view of the expansion \eqref{eq:exp}, we consider nonlinearity of the form $|u|^{5/3-n}u^n$.
To this end, we introduce a Lipschitz $\mu$ norm $(\mu>1)$. 
%Denote $\pa_1 = \pa_{z}$ and $\pa_2=\pa_{\overline{z}}$.
For a multi-index $\a=(\a_1,\a_2)\in(\Z_{\ge0})^2$, define $\pa^{\a}=\pa_z^{\a_1} \pa_{\overline{z}}^{\a_2}$.
Put $\mu = N + \b$ with $N \in \Z$ and $\b \in (0,1]$. For a function $G \in C^{N}( \R^2, \C)$, we define
\begin{align*}
\norm{G}_{\Lip \mu} = \sum_{|\a| \le N-1} \sup_{z \in \C \setminus \{0\}} \frac{|\pa^{\a}G(z)|}{|z|^{\mu - |\a|}} + \sum_{|\a| = N}\sup_{z \neq z'} \frac{|\pa^{\a}G(z)-\pa^{\a}G(z')|}{|z-z'|^{\b}}.
\end{align*}
If $G \in C^{N}(\R^2, \C)$ and $\norm{G}_{\Lip \mu} < \I$, then we write $G \in \Lip \mu$.

%
% Let us evaluate $\tnorm{|z|^{\frac53-n}z^n}_{\Lip \frac53}$.
% give us an upper bound of order $O(n^{5/3})$.

\begin{lemma} \label{lem4:5a}
$\tnorm{|z|^{\frac53-n}z^n}_{\Lip \frac53} \le C\Jbr{n}^\frac53$
for some $C>0$ and for any $n\in \Z$.
\end{lemma}

\begin{proof}
Set $F(z)=|z|^{5/3-n}z^n$.
By definition of the Lipschitz norm,
\begin{align*}
	\norm{F}_{\Lip \frac53} = &\sup_{z \in \C \setminus \{0\}} \frac{|F(z)|}{|z|^{5/3}} %+ \sup_{z \in \C \setminus \{0\}} \frac{|F_{z}(z)|}{|z|^{2/3}} + \sup_{z \in \C \setminus \{0\}} \frac{|F_{\bar{z}}(z)|}{|z|^{2/3}} \\
	&+ \sup_{z \neq w} \frac{|F_{z}(z)-F_{z}(w)|}{|z-w|^{2/3}} + \sup_{z \neq w} \frac{|F_{\bar{z}}(z)-F_{\bar{z}}(w)|}{|z-w|^{2/3}}.
\end{align*}
Obviously, the first term is bounded.
In what follows, we estimate the second term. The third term is handled similarly.

Introduce $\wt{F}(z)$ by
$$
F_{z}(z) = \(\frac{5}{6}+\frac{n}{2}\)z^{-\frac{1}6+\frac{n}2}\bar{z}^{\frac{5}{6}-\frac{n}2}
=: \(\frac{5}{6}+\frac{n}{2}\)\wt{F}(z).
$$
To estimate the second term, it suffices to consider the case $w=1$.
Indeed, if $w=0$ then
\[
	\frac{|\wt{F}(z)-\wt{F}(w)|}{|z-w|^{2/3}} =  \frac{|\wt{F}(z)|}{|z|^{2/3}} \le C,
\]
otherwise, denoting $z$ and $w$ in the phase amplitude form $z = |z|e^{i\t}$ and $w = |w|e^{i\tau}$,
we have
\begin{align*}
	\frac{|\wt{F}(z)-\wt{F}(w)|}{|z-w|^{2/3}} 
	&= \frac{||z|^{2/3}e^{i(n-1)\t}-|w|^{2/3}e^{i(n-1)\tau}|}{||z|e^{i\t}-|w|e^{i\tau}|^{2/3}} \\
	&= \frac{|\(\frac{|z|}{|w|}\)^{2/3}e^{i(n-1)(\t-\tau)}-1|}{|\frac{|z|}{|w|}e^{i(\t-\tau)}-1|^{2/3}} 
	= \frac{|\wt{F}(\wt{z})-\wt{F}(1)|}{|\wt{z}-1|^{2/3}},
\end{align*}
where $\wt{z}=z/w$.
Let $\e \in (0,1)$ to be chosen later. Using the elemental inequality $|z-1| \ge \max (|z-1|, |z|-1)$, we have
\begin{align*}
	\frac{||z|^{2/3}e^{i(n-1)}-1|}{|z-1|^{2/3}} 
% 	&\le \frac{|z|^{2/3}+1}{\( \max (|z-1|, |z|-1)\)^{3/5}} \\
	\le \frac{|z|^{2/3}+1}{\( \max (\e, |z|-1)\)^{2/3}} \le C(\e)
\end{align*}
for any $|z-1|>\e$.
Let us consider tha case $|z-1| \le \e$.
By the Taylor expansion, if $\e$ is sufficiently small then
$|re^{i\t}-1| \ge C(|r-1|+|\t|)$
for any $|z-1| \le \e$,
which implies $||z|e^{i\t}-1|^{2/3} \ge C(\left| |z|-1 \right|^{2/3}+|\t|^{2/3})$.
Hence,
\begin{align*}
	\frac{||z|^{2/3}e^{i(n-1)\t}-1|}{|z-1|^{2/3} } &\le C \frac{| |z|^{2/3}-1| + |z|^{2/3} |e^{i(n-1)\t}-1|}{| |z|-1 |^{2/3}+|\t|^{2/3}} \\
	&\le C \frac{| |z|^{2/3}-1| + |z|^{2/3} |(n-1)\t|^{2/3}}{| |z|-1 |^{2/3}+|\t|^{2/3}} 
% 	&\le C \Jbr{n}^{2/3}\frac{| |z|^{2/3}-1| + |z|^{2/3} |\t|^{2/3}}{| |z|-1 |^{2/3}+|\t|^{2/3}} \\
	\le C\Jbr{n}^{2/3}
\end{align*}
for any $|z-1| \le \e$, where we have used $|e^{i\tau}-1|=2|\sin (\tau/2)|\le 2^{1/3}|\tau|^{2/3}$. 
Thus, combining the above estimates, we see that
\[
	\sup_{z \neq w} \frac{|F_{z}(z)-F_{z}(w)|}{|z-w|^{2/3}} \le C\Jbr{n}^{5/3}. 
\]
This completes the proof.
\end{proof}

We recall the fractional chain rule in \cite[Theorem 5.3.4.1]{MR1419319} (see also \cite{MS}). 
\begin{lemma} \label{lem4:4a}
Suppose that $\mu >1$ and $s \in (0,\mu)$. Let $G \in \Lip \mu$. Then, there exists a positive constant $C$ depending on $\mu$ and $s$ such that 
\[
\| |D_x|^{s} G(f) \|_{L^{2}_{x}} \le C\|G\|_{\Lip \mu}\|f\|_{L^{\I}_{x}}^{\mu-1}\||D_x|^{s} f\|_{L^{2}_{x}}
\]
holds for any $f\in L^\I \cap \dot{H}^s$.
\end{lemma}

\subsection{Estimates on nonlinearity}

We give some specific estimates on $\wha{w}$ and $|\wha{w}|^{5/3-n} \wha{w}^n$ by using the tools established 
in the preceding subsection.

\begin{lemma} \label{lem1:2a}
Let $3/2 \le \d< \d' < 5/3$.
Let $u_+\in H^{0.5/3}$ and define
$\wha{w}$ as in \eqref{def:upw}. 
Then, 
\begin{align*}
	\norm{\wha{w}}_{H^{\d}} &\le C\norm{u_+}_{H^{0, \d}}
	\Jbr{ \norm{u_+}_{H^{0, \d}} }^{\frac23} \Jbr{ g_1 \norm{\wha{u_+}}_{L^\I}^{\frac1{3}} \log t}^2,
\end{align*}
and
\begin{multline*}
	\norm{|\wha{w}|^{\frac53 - n}{\wha{w}}^{n}}_{{H}^{\d}} \le C \Jbr{n}^{\d'}
	 \norm{\wha{u_+}}_{L^\I}^{\frac23} 
	\norm{u_+}_{H^{0, \frac53}} \\\times \Jbr{ \norm{u_+}_{H^{0, \frac53}} }^{\frac23}  \Jbr{ g_1\norm{\wha{u_+}}_{L^\I}^{\frac13} \log t}^{2}
\end{multline*}
for any $t\ge 2$.
\end{lemma}

\begin{proof}%[Proof of Lemma \ref{lem1:2a}]
Let us prove the first estimate. Since the $L^2$ estimate is trivial, we estimate $\dot{H}^{\d}$ norm.
Fix $t\ge3$ and
let $\l= -g_1 \log t$ for simplicity.
Let $\P(z)=\exp(i\l|z|^{2/3})$. 
Note that $\P(z)$ is a $2/3$-H\"older functions with norm $O(|\l|)$ because
\[
	|\P(z_1)-\P(z_2)|
	= \abs{\sin \(\frac{\l}2(|z_1|^{2/3}-|z_2|^{2/3})\)} 
	\le C |\l| |z_1-z_2|^{2/3}.
\]
It holds that
\[
	\norm{\wha{w}}_{\dot{H}^{\d}}
	\le C\norm{(\nabla \wha{u_+})\P(\wha{u_+})}_{\dot{H}^{\d-1}}
	+ C|\l|\norm{F(\wha{u_+})(\nabla \wha{u_+})\P(\wha{u_+})}_{\dot{H}^{\d-1}},
\]
where $F(x)=z (\frac{d}{dz}+\frac{d}{d\overline{z}}) |z|^{2/3}$ is a $2/3$-H\"older continuous function.
We only estimate the second term since the first term is treated in a similar way.
It follows that
\begin{align*}
	&{}\norm{F(\wha{u_+})(\nabla \wha{u_+})\P(\wha{u_+})}_{\dot{H}^{\d-1}} \\
	\le{}& C\tnorm{|\nabla|^{\d-1}F(\wha{u_+})}_{L^{\frac3{\d-1}}} \norm{\nabla \wha{u_+}}_{L^{\frac6{5-2\d}}} \norm{\P(\wha{u_+})}_{L^\I} \\
	&{}+ C\tnorm{F(\wha{u_+})}_{L^\I} \tnorm{|\nabla|^{\d-1}\nabla \wha{u_+}}_{L^2} \norm{\P(\wha{u_+})}_{L^\I} \\
	&{}+ C\tnorm{F(\wha{u_+})}_{L^\I} \norm{\nabla \wha{u_+}}_{L^{\frac6{5-2\d}}} \tnorm{|\nabla|^{\d-1}\P(\wha{u_+})}_{L^{\frac3{\d-1}}}.
\end{align*}
% \footnote{$1-\frac3p \le \d -\frac32 \Leftrightarrow p \le \frac{6}{5-2\d}$ より$H^{\d} \hookrightarrow L^{\frac{6}{5-2\d}}$.}
Obviously, the second term is bounded by $\norm{\wha{u_+}}_{L^\I}^{2/3}\norm{\wha{u_+}}_{\dot{H}^{\d}}$.
By \cite[Proposition A.1]{MR2318286}, 
\[
	\tnorm{|\nabla|^{\d-1}\P(\wha{u_+})}_{L^{\frac3{\d-1}}} \le C|\l|\norm{\wha{u_+}}_{L^\I}^{\frac56-\frac\d2}
	\tnorm{|\nabla|^{s}\wha{u_+}}_{L^{3/s}}^{-\frac16+\frac\d2} 
% 	\lesssim |\l|\norm{\phi}_{L^\I}^{\frac56-\frac\d2}\tnorm{\phi}_{\dot{H}^{3/2}}^{-\frac16+\frac\d2}
	\le C|\l|\norm{\wha{u_+}}_{H^\d}^{\frac23},
\]
% \footnote{$s-\frac{s}{3} \times 3 < \d -\frac32$より$H^{\d} \hookrightarrow L^{3/s}$. }
where $s=(\d-1)/(\frac12(\frac23 + (\d-1))) \in (\frac32(\d-1),1)$. Hence, the the third term is bounded by
$ C|\l|\norm{\wha{u_+}}_{L^\I}^{2/3} \norm{\wha{u_+}}_{H^\d}^{5/3}$.	
Since $F$ is $2/3$-H\"older, the same argument shows that the first term is bounded by
$C\norm{\wha{u_+}}_{H^\d}^{5/3}$, which completes the proof of the first estimate.	

Let us show the second. Let $\e>0$ be chosen later.
By interpolation inequality, H\"older's inequality, Lemma \ref{lem4:4a} and Lemma \ref{lem4:5a}, we have
\begin{align*}
	\norm{|\wha{w}|^{1+\frac23 - n}{\wha{w}}^{n}}_{{H}^{\d}}
	&\le \norm{|\wha{w}|^{\frac53 - n}{\wha{w}}^{n}}_{L^2}^{1-\t}
	\norm{|\wha{w}|^{\frac53 - n}{\wha{w}}^{n}}_{{H}^{\frac53 -\e}}^{\t} \\
	&\le C_\eps \Jbr{n}^{\frac53 \t} \norm{\wha{w}}_{L^\I}^{\frac23} \norm{\wha{w}}_{L^2}^{1-\t} \norm{\wha{w}}_{H^{\frac53 -\e}}^{\t}
\end{align*}
as long as $\d < \frac53 - \e$,
where $\t %= \(\frac53 - \e\)^{-1} \d
=\frac35 (1+\frac{3\e}{5-3\e})\d$.
Choose $\e>0$ so small that $\frac53 \t \le \d' $.
Then the second estimate is a consequence of the first. %with $\d = \frac35 -\widetilde{\e}$.
\end{proof}

The following estimate is shown as in \cite{MM2}.

\begin{lemma} \label{lem1:2}
Let $\wha{w}$ be as in \eqref{def:upw}.
Then, it holds that
% \begin{align*}
% \norm{\wha{w}}_{H^{2}} &\le C\norm{u_+}_{H^{0, 2}}\Jbr{ g_1 \norm{\wha{u_+}}_{L^\I}^{\frac2{3}} \log t}^2,\\
% \norm{\partial_t\wha{w}}_{H^{2}} &\le C\frac{|g_1|}{t}\norm{\wha{u_+}}_{L^\I}^{\frac2{3}} \norm{u_+}_{H^{0, 2}}\Jbr{ g_1 \norm{\wha{u_+}}_{L^\I}^{\frac2{3}} \log t}^2.
% \end{align*}
% Moreover, 
\begin{align*}
%	\norm{|\wha{w}|^{1+\frac2d - n}{\wha{w}}^{n}}_{{H}^{\d}}
%	&\le C \Jbr{n}^\d
%	 \norm{\wha{u_+}}_{L^\I}^{\frac2d} 
%	\norm{u_+}_{H^{0, 2}}\Jbr{ g_1\norm{\wha{u_+}}_{L^\I}^{\frac2d} \log t}^\d\\
	\norm{\partial_t (|\wha{w}|^{\frac53 - n}{\wha{w}}^{n})}_{{H}^{\d}}
	&\le C\frac{\Jbr{n}^{1+\d}|g_1|}{t}\norm{\wha{u_+}}_{L^\I}^{\frac4{3}} \norm{u_+}_{H^{0, 2}}
	\Jbr{ g_1 \norm{\wha{u_+}}_{L^\I}^{\frac2{3}} \log t}^\d
\end{align*}
for any $0 \le  \d \le 2$ and $t\ge2$. 
\end{lemma}

 \begin{remark}
The function $\partial_t (|\wha{w}|^{\frac53 - n}{\wha{w}}^{n})$ is of the form
$$t^{-1} F_n(\wha{u_+}) \exp (-ing_1|\wha{u_+}|^{2/3}\log t),$$
where $F_n$ satisfies $|F^{(j)}_n(z)| \le C \Jbr{n}^{1+j} |z|^{\frac73-j}$ for $j=0,1,2$.
Therefore, we can estimate its $H^2$-norm by an explicit calculation.
Then, the estimate follows from an interpolation as in \cite{MM2}.
It is possible to estimate this term in a similar way to Lemma \ref{lem1:2a}.
This improves the assumption on $\wha{u_+}$ into $\wha{u_+} \in H^\d$ but 
the order of $|n|$ becomes worse.
This is the reason why we apply an interpolation argument to this term, as in \cite{MM2}.
The full regularity $\wha{u_+} \in H^2$ is required in this step.
 \end{remark}

\section{Construction of a solution around given asymptotic profile}\label{sec:abst}
In this section, we solve an equation of the form
\begin{equation}\label{eq:abst}
	u(t) -u_p(t) 
	= i \int_t^{\I} U(t-s) \( F (u) - F (u_p) \)(s) ds +\mathcal{E}(t),
\end{equation}
where $u_p$ is a given asymptotic profile of the form \eqref{eq:masymptotic2}
and $\mathcal{E}(t)$ is an external term.
Remark that our equation \eqref{eq:inteq} is of the form.
\begin{proposition}\label{prop:abst}
Suppose that $g$ is Lipschitz continuous. 
Let $\wh{u_+} \in L^\I$ and let $u_p$ be as in \eqref{eq:masymptotic2}.
There exists a constant $\eps_0 = \eps_0(\norm{g}_{\mathrm{Lip}})>0$ such that
if $\norm{\wh{u_+}}_{L^\I} \le \eps_0$
and if an external term $\mathcal{E}$ satisfies
$\norm{\mathcal{E}}_{X_{T_0,b}} \le M$ for some $T_0\ge 1$, $M>0$, and $b>3/4$,
then \eqref{eq:abst} admits a unique solution $u(t)$ in $X_{T_1,b,2M}$ for some $T_1= T_1(M,\norm{g}_{\mathrm{Lip}},b) \ge T_0$.
Moreover, for any function $\mathcal{V}$, admissible pair $(q,r)$, and $\tilde{b}\le b$, the solution satisfies
\[
	\sup_{t\ge T_1} t^{\tilde{b}} \norm{u-u_p-\mathcal{V}}_{L^q_t([t,\I);L^r_x)}
	\le C + \sup_{t\ge T_1} t^{\tilde{b}} \norm{\mathcal{E}-\mathcal{V}}_{L^q_t([t,\I);L^r_x)}.
\]
\end{proposition}

The proposition shows that the conclusion of Theorem \ref{thm:main} follows from
the estimate \eqref{eq:main_step_of_proof}, which is true for $b<\delta/2$ in view of Proposition \ref{prop:main}.
Indeed, for each $3/4<b<\delta/2$, we can construct a solution $u(t,x)=u(t,x;b)$ on $[T_1(b),\I)$
which satisfies \eqref{eq:mainest} for this $b$, by using the proposition.
Uniqueness property of the proposition then show these solution coincide each other.
Hence, with a help of the standard well-posedness theory in $L^2$,
the solution exists in an interval independent of $b$, say $[T_1,\I)$, and
satisfies \eqref{eq:mainest} for any $b<\d/2$.
The estimate \eqref{eq:mainest2} in Theorem \ref{thm:main2} follows from corresponding estimate on $\mathcal{E}_{\mathrm{r}}+\mathcal{E}_{\mathrm{nr}}$ given in Proposition \ref{prop:main}.

\begin{lemma} \label{non1:g2}
Suppose that $g$ is Lipschitz continuous. 
Let $\wh{u_+} \in L^\I$ and let $u_p$ be as in \eqref{eq:masymptotic2}.
If $b > 3/10$ then it holds that
\begin{multline*}
\norm{\int_t^{\I}U(t-s) \( F(v) - F(u_p)\) ds}_{X_{T,b}} \\
\le C\norm{g}_{\mathrm{Lip}} \norm{v-u_p}_{X_{T,b}} \( \norm{v-u_p}_{X_{T,b}}^{\frac{2}{3}}T^{\frac{1}{2} -\frac{2}{3}b} + \Lebn{\wha{u_+}}{\I}^{\frac{2}{3}} \)
\end{multline*}
for any $v \in X_{T,b,R}$ with $T\ge 1$ and $R>0$.
\end{lemma}
\begin{remark}
The constant $C$ in the estimate of the above lemma can be taken independent of $b$, provided $b\ge 3/4$.
\end{remark}

\begin{proof}
The estimate 
 is the same as in \cite{HNST,HWN,ShT} except for using the endpoint Strichartz' estimate.
Let us first decompose $F(v) - F(u_p) = F^{(1)}(v) + F^{(2)}(v)$, where
\begin{align*}
	F^{(1)}(v) = \chi_{\{ |u_p| \le \left| v-u_p \right| \}} \( F(v) - F (u_p)\), \\
	 F^{(2)}(v) = \chi_{\{ |u_p| \ge \left| v-u_p \right| \}}\( F(v) - F (u_p)\),
\end{align*}
and $\chi_{A}$ is a characteristic function on $A \subset \R^{1+3}$.
Since $g$ is Lipschitz, it follows from \cite[Appendix A]{MM2} that
\begin{align*}
	&\left| F(v) - F(u_p) \right| \le C\norm{g}_{\mathrm{Lip}}\( \left| v-u_p \right|^{1+ \frac{2}{3}} + |u_p|^{\frac{2}{3}} \left|v - u_p\right| \).
\end{align*}
Since $b>3/10$, we estimate $F^{(1)}(v)$ by the endpoint Strichartz estimate as follows: 
\begin{align*}
	\norm{\int_t^{\I}U(t-s) F^{(1)}(v) ds}_{L^{\I} (T, \I ; L^2)}
	&\le C \norm{|v - u_p|^{1+\frac{2}{3}}}_{L^{2} (T, \I ; L^{\frac{6}{5}})} \\
	&\le C T^{\(\frac12-\frac{2}{3}b\)-b}\norm{v-u_p}_{X_{T,b}}^{\frac53}.
\end{align*}
For estimate of $F^{(2)}(v)$, 
we use $\norm{u_p(t)}_{L^\I} \le C t^{-3/2}\norm{\wh{u_+}}_{L^\I}$.
Then,
\begin{align*}
	\norm{\int_t^{\I}U(t-s) F^{(2)}(v) ds}_{L^{\I} (T, \I ; L^2)}
	&\le C \norm{|u_p|^{\frac{2}{3}}|v - u_p|}_{L^{1} (T, \I ; L^{2})}\\
	&\le C T^{-b}\norm{v-u_p}_{X_{T,b}}\norm{\wh{u_+}}_{L^\I}^{\frac23}
\end{align*}
as long as $b>0$.
This completes the proof.
\end{proof}

\begin{proof}[Proof of Proposition \ref{prop:abst}]
Let 
\begin{equation}\label{def:Phi2}
\P(v)(t) := u_p(t) + i \int_t^{\I} U(t-s) \( F (v) - F (u_p)  \)(s) ds  +\mathcal{E}(t)
\end{equation}
By Lemma \ref{non1:g2} and by assumption, %Lemma \ref{non1:g1}, and Proposition \ref{non:im1}, we have
we have
\begin{align}
\norm{\P(v)-u_p}_{X_{T_0, b}} 
\le C_1 \norm{g}_{\mathrm{Lip}} R \(R^{\frac{2}{3}}T^{\frac{1}{2}-\frac{2}{3}b}+\e_0^{\frac23} \) 
+M
\label{est1:t1} 
\end{align}
for any $v \in X_{T, b, R}$ with $T\ge T_0$ and $R>0$.
We next see that 
\begin{equation}
d(\P(v_1), \P(v_2)) \le C_2 \norm{g}_{\mathrm{Lip}}  \( R^\frac{2}{3} T^{\frac12 - \frac{2}{3}b} + \e_0^\frac{2}{3} \) d(v_1, v_2) \label{est1:t2}
\end{equation}
for any $v_1,v_2 \in X_{T, b, R}$  with $T\ge 1$ and $R>0$.
Indeed, by the integral equation of (NLS), we see that
\begin{align*}
\P(v_1) -\P(v_2) = i \int_t^{\I} U(t-s) \( F (v_1) - F (v_2) \)(s) ds.
\end{align*}
One finds 
\begin{align*}
\left| F (v_1) - F (v_2) \right| 
&{}\le C \norm{g}_{\mathrm{Lip}} ( |v_1|^{\frac{2}{3}} + |v_2|^{\frac{2}{3}} )|u-v| \\
&{}\le C \norm{g}_{\mathrm{Lip}} ( |v_1-u_p|^{\frac{2}{3}} + |v_2-u_p|^{\frac{2}{3}} )|u-v| \\
&{}\quad + C \norm{g}_{\mathrm{Lip}} |u_p|^{\frac{2}{3}}|v_1-v_2|.
\end{align*}
Motivated by the calculation, we introduce a decomposition of $F(v_1)-F(v_2)$ into two parts 
depending on whether $|v_1-u_p|^{\frac{2}{3}} + |v_2-u_p|^{\frac{2}{3}} \ge |u_p|^{\frac23}$ or not.
The rest of the proof is similar to that of Lemma \ref{non1:g2}.

% We are now in a position to complete the proof of Theorem \ref{thm:main}.
Choose $\e_0=\e_0(\norm{g}_{\mathrm{Lip}})$ so small that
\[
	C_1 \norm{g}_{\mathrm{Lip}} \e_0^{\frac23} \le \frac14,\quad
	C_2 \norm{g}_{\mathrm{Lip}} \e_0^{\frac23} \le \frac14.
\]
Choose $R=2M$.
By the assumption $b>3/4$, we can choose
$T_1\ge T_0$ such that $(2M)^{\frac23}T_1^{\frac12-\frac23b}\le \e_0^{\frac23}$.
It then follows from \eqref{est1:t1} and \eqref{est1:t2} that
\[
	\norm{\P(v)-u_p}_{X_{T_1, b}} \le (4 C_1 \norm{g}_{\mathrm{Lip}} \e_0^{\frac23} + 1)M \le 2M = R 
\]
and
\[
	d(\P(v_1), \P(v_2)) \le  2 C_2 \norm{g}_{\mathrm{Lip}} \e_0^{\frac23}d(v_1,v_2) \le \frac12 d(v_1,v_2)
\]
for any $v_1,v_2 \in X_{T_1, b, 2M}$,
which shows $\P: X_{T_1, b, 2M} \rightarrow X_{T_1, b, 2M}$ is a contraction mapping. 
Thus, we obtain a unique solution $u(t)\in X_{T_1,b,2M}$ to \eqref{eq:abst}.

Take $\tilde{b}\le b$ and an admissible pair $(q,r)$.
Then, as in Lemma \ref{non1:g2}, we deduce from
the Strichartz estimate that
\[
	t^{\tilde{b}} \norm{u-u_p-\mathcal{V}}_{L^q_t([t,\I);L^r_x)}
	\le C t^{\tilde{b}-b}(2M) + t^{\tilde{b}} \norm{\mathcal{E}-\mathcal{V}}_{L^q_t([t,\I);L^r_x)}
\]
for any $t\ge T_1$.
This shows the latter statement.
\end{proof}

\section{Proof of main results} \label{sec:mainlem}

In this section, we prove main theorems by showing Proposition \ref{prop:main}.
Let us first recall an estimate in \cite[Lemma 2.1]{HWN} which shows $\mathcal{E}_{\mathrm{r}}$ is harmless.
\begin{lemma} \label{non1:g1}
Let $\d\in (3/2,5/3)$.
For any $u_+ \in H^{0,5/3}$,
there exists a constant $C=C(g_1,\norm{u_+}_{H^{0,5/3}})$ such that
\[
\norm{\mR (t) \wha{w}}_{L^{\I}_t ([T, \I) ; L^2) \cap L^{2}_t ([T, \I) ; L^6)} 
\le C T^{-\frac\d2} (\log T)^2
\]
and
\[
\norm{\int_t^{\I}U(t-s)\mR(s) \mG (\wha{w}) \frac{ds}{s}}_{L^{\I}_t ([T, \I) ; L^2) \cap L^{2}_t ([T, \I) ; L^6)}  \\
\le C T^{-\frac\d2} (\log T)^3
\]
hold for all $T\ge2$.
\end{lemma}
% \begin{proof}(後でコメントアウト予定)
% Denote by $(q,r)$ a pair $(\I,2)$ or $(2,6)$.
% By the Sobolev embedding and the property of $D(t)$, we have
% \begin{align*}
% 	\Lebn{\mR(t) \wha{w}}{r} 
% 	&= Ct^{-\frac2q}\Lebn{|\n|^{\frac{2}{q}} \(U\(-\frac{1}{4t}\) -1\)\wha{w}}{2} 
% 	\le Ct^{-\frac{2}q} \Lebn{|\n|^{\frac{2}{q}}\wha{w}}{2},
% \end{align*}
% The case $b \le 1/q$ immediately follows from the above estimate and the relation $2/q \le1 < \delta_0$,
% and Lemma \ref{lem1:2a}.
% In the case $b>1/q$, it follows that
% \[
% 	\Lebn{\mR(t) \wha{w}}{r} \le Ct^{-\frac2q} \Lebn{\left| \frac{|\n|^2}{t}\right|^{b-\frac1{q}} |\n|^{\frac2q}\wha{w} }{2} 
% 	\le Ct^{-\frac1q - b} \hSobn{\wha{w}}{2b}.
% \]
% Then, the result follows from Lemma \ref{lem1:2a}.

% The second estimate follows in a similar way.
% As in the first estimate,
% \begin{align*}
% 	&\norm{\int_t^{\I}U(t-s)\mR(s) \mG (\wha{w}) \frac{ds}{s}}_{L^{\I}_t (T, \I ; L^2) \cap L^{2}_t (T, \I ; L^6)} \\
% 	\le{}& C \int_T^\I \Lebn{\mR(t) \mG(\wha{w})(s)}{2} \frac{ds}s \\
% 	\le{}& C \int_T^\I s^{-b} \Sobn{\mG(\wha{w})(s)}{2b} \frac{ds}s 
% \end{align*}
% By Lemmas \ref{lem1:2a}, \ref{lem4:4a}, and \ref{lem4:5a},
% \begin{align*}
% 	\hSobn{\mG(\wha{w})(s)}{2b}
% 	&\le C |g_1|\norm{\wha{w}(s)}_{L^\I}^\frac23 \norm{\wha{w}(s)}_{\dot{H}^{2b}}\\
% 	&\le C |g_1|\norm{\wha{u_+}}_{L^\I}^{\frac23}
% 	\norm{u_+}_{H^{0, \d_0}}
% 	\Jbr{ \norm{u_+}_{H^{0, \d_0}} }^{\frac23} \Jbr{ g_1 \norm{\wha{u_+}}_{L^\I}^{\frac1{3}} \log s}^2 \qedhere
% \end{align*}
% \end{proof}

Hence, we concentrate on the treatment of $\mathcal{E}_{\mathrm{nr}}$ in what follow.
As for this term, we have the following.
\begin{proposition} \label{non:im1}
Let $3/2 < \d < 5/3$. 
Assume that $\sum_{n\in\Z}|n|^{1+\eta}|g_n| < \I$
for some $\eta> \frac12(\delta-\frac{3}2)$.
Let $\mathcal{V}$ and $v_p$ be as in \eqref{def:calV} and \eqref{def:vp}, respectively.
For any $u_+ \in H^{0,2} \cap {H}^{-\d} $, there exists a constant $C=C(g_1,\norm{u_+}_{H^{0,2} \cap {H}^{-\d}})$ such that
\begin{equation}\label{eq:nrV}
	\norm{\mathcal{E}_{\mathrm{nr}} - \mathcal{V}}_{L^{\I}_t (T, \I ; L^2)\cap L^2(T,\I;L^6)} 
	\le CT^{-\frac{\d}2} \Jbr{\log T}^3 \sum_{n \ne 0,1} |n|^{1+\eta} |g_n|
\end{equation}
holds for all $T\ge 2$.
Moreover, $\mathcal{V}$ is small in $L^\I (T,\I;L^2)$ in such a sense that 
\begin{equation}\label{eq:Vest}
	\norm{\mathcal{V}}_{L^\I (T,\I;L^2)} \le C \norm{u_+}_{H^{0, \frac53}\cap H^{-\d}}^{\frac53} T^{-\frac\d2} \sum_{n\neq0,1} |n|^{-\d}|g_n|
\end{equation}
for $T\ge2$.
Furthermore, $\mathcal{V}$ is approximated by $v_p$ in $L^2 (T,\I;L^6)$:
There exists $C=C(g_1,\norm{u_+}_{H^{0, \frac53}\cap H^{-\d}})>0$ such that
\begin{equation}\label{eq:Vest2}
	\norm{\mathcal{V}-v_p}_{L^2 (T,\I;L^6)} \le 
	C T^{-\frac\d2}(\log T)^3 \sum_{n\neq0,1} |n|^\frac56|g_n|.
\end{equation}
for $T\ge2$.
\end{proposition}
The estimates \eqref{eq:nrV} and \eqref{eq:Vest} complete the proof of Proposition \ref{prop:main}.
The estimates \eqref{eq:Vest} and \eqref{eq:Vest2} imply \eqref{eq:behavior1} and \eqref{eq:behavior2}, respectively.
Hence, Theorems \ref{thm:main} and \ref{thm:main2} both follow from the above proposition.

\subsection{Integration by parts and extraction of the main part}

Without loss of generality, we may suppose that $b\ge3/4$.
Using $u_p  = M(t)D(t)\wha{w}(t) = D(t)E(t)\wha{w}(t)$ with $E(t) = e^{it|x|^2}$, we obtain
\[
	\mN (u_p) = \sum_{n \ne 0,1} g_n \( \frac{1}{2t} D(t) i^{-\frac32(n-1)} E^n(t) \phi_n (t) \),
\]
where 
\begin{equation*}
	\phi_n(t):= |\wha{w}(t)|^{\frac53-n}\wha{w}^n(t).
\end{equation*}
Let $\psi_0(x) = e^{-|x|^2/4} \in \m{S}$ and set
$\m{K}(t,n):=\m{K}_{\psi_0}(t,n)$ as in \eqref{def:Kpsi}.
Remark that $\nabla \psi_0 (0) =0$.
% with $\s = 1$ if $d=1$ and $\s = \frac{2+\d}{3}>1$ if $d=2$.
We decompose $\mN(u_p)$ into low frequency part and high frequency part,
\[
	\mN(u_p) = \m{P} + \m{Q}, 
\]
where
\begin{align*}
	\m{P} &= \sum_{n \ne 0,1} g_n \( \frac{1}{2t} D(t)\( i^{-\frac32(n-1)}E^n(t) \m{K} \phi_n(t) \) \), \\
	\m{Q} &= -\sum_{n \ne 0,1} g_n \( \frac{1}{2t} D(t)\( i^{-\frac32(n-1)}E^n(t) (\m{K}-1) \phi_n(t) \) \).
\end{align*}

As for the high frequency part $\m{Q}$, we have the following.
% \begin{lemma} Fix $4/3 \le b \le \d/2$ and $\eps>0$.
% For any $T\ge 2$, 
% \begin{multline}
% 	\norm{\int_{t}^\I U(t-s) \m{Q}(s) ds}_{L^\I(T,\infty;L^2) \cap L^2(T,\infty;L^6)} \\ 
% 	\le CT^{-b}\Jbr{ |g_1|^2\norm{\wha{u_+}}_{L^\I}^{\frac23} (\log T)^3}
% 	\Jbr{ \norm{u_+}_{H^{0, \frac53}} }^{\frac73}  \sum_{n \ne 0,1} \Jbr{n}^{\e }|g_n|. 
% 	\label{non:n4}
% \end{multline}
% \end{lemma}
\begin{lemma} Fix $\eps>0$. 
There exists a constant $C=C(g_1,\norm{u_+}_{H^{0, \frac53}})$ such that
\begin{multline}
	\norm{\int_{t}^\I U(t-s) \m{Q}(s) ds}_{L^\I(T,\infty;L^2) \cap L^2(T,\infty;L^6)} \\ 
	\le CT^{-\frac\d2}(\log T)^3 \sum_{n \ne 0,1} |n|^{\e }|g_n|. 
	\label{non:n4}
\end{multline}
for any $T\ge 2$.
\end{lemma}
\begin{proof}
By Strichartz' estimate, it suffices to bound $\norm{\m{Q}}_{L^1(T,\infty;L^2)}$.
% Further, it suffices to consider the case $b=\d/2$.
By using Lemma \ref{mol:1.1} (ii) and Lemma \ref{lem1:2a}, we have
\begin{align*}
	\Lebn{\m{Q}(t)}{2} \le{}& Ct^{-1} \sum_{n \ne 0,1}|g_n| \Lebn{(\m{K}-1) \phi_n}{2} \\
	\le{}& Ct^{-1-\frac\d2} \sum_{n \ne 0,1} |n|^{-\d} |g_n| \norm{\phi_n}_{\dot{H}^{\d}} \\
	\le{}& Ct^{-1-\frac\d2}
	 \norm{\wha{u_+}}_{L^\I}^{\frac23} 
	\norm{u_+}_{H^{0, \frac53}} \Jbr{ \norm{u_+}_{H^{0, \frac53}} }^{\frac23}  \\
	&{}\times \Jbr{ g_1\norm{\wha{u_+}}_{L^\I}^{\frac13} \log t}^2
	\sum_{n \ne 0,1} \Jbr{n}^{\e}|g_n|
\end{align*}
for any $\e>0$.
\end{proof}

Next, we consider the low-frequency part. % $\norm{\int_t^{\I}U(t-s) \m{P} (s) ds}_{L^2}$. 
By the factorization of $U(t)=M(t)D(t)\F M(t)$,
we see that
\begin{equation}\label{eq:target}
	\int_t^{\I}U(t-s) \m{P}(s) ds = U(t)\F^{-1} \int_t^{\I}\F U(-s) \m{P}(s) ds.
\end{equation}
Again by factorization of $U(t)$, we have 
\begin{equation}\label{eq:factorize}
	\F U(-s) D(s) E^{\rho}(s) = i^{\frac{3}{2}} E^{1-\frac{1}{\rho}}(s)U\( \frac{\rho}{4s} \) D\(\frac{\rho}{2}\)
\end{equation}
for $\rho \ne 0$ (see \cite{HWN}). Therefore, we further compute
\begin{align*}
	\F U(-s) \m{P}(s) &= \sum_{n \ne 0,1} i^{-\frac32(n-1)} g_n \frac1{2s} \F U(-s)D(s) E^{n}(s) \m{K}\phi_n (s) \\
	&= \sum_{n \ne 0,1} i^{-\frac32(n-2)} g_n \frac{1}{2s} E^{1-\frac{1}{n}}(s) U\( \frac{n}{4s} \) D\(\frac{n}{2}\) \m{K} \phi_n(s).
\end{align*}
Now, we have $E^{1-\frac{1}{n}}(s) = A_n(s) \pa_s (s E^{1-\frac{1}{n}}(s))$ for $n\neq0,1$,
where 
\begin{equation}\label{der:An}
A_n(s) := \( 1+ i\(1-\frac{1}{n}\)s|x|^2 \)^{-1}.
\end{equation}
Further,
\[
	\pa_{s}U \( \frac{n}{4s} \) = U\( \frac{n}{4s}\) \( \pa_s - \frac{i n}{2s^2} \D \).
\]
Therefore, an integration by parts gives us
\begin{align}
\begin{aligned}
	&\int_t^{\I} E^{1-\frac{1}{n}}(s) U\( \frac{n}{4s} \) D\(\frac{n}{2}\) \m{K}\f_n(s) \frac{ds}{s} \\
	={}& -E^{1-\frac{1}{n}}(t) A_n(t) U\( \frac{n}{4t} \) D\(\frac{n}{2}\) \m{K}\f_n(t) \\
	&{}- \int_t^{\I} E^{1-\frac{1}{n}}(s) s \pa_s \(s^{-1}A_n(s)\) U\( \frac{n}{4s} \) D\(\frac{n}{2}\) \m{K}\f_n(s) ds \\
	&{}- \int_t^{\I} E^{1-\frac{1}{n}}(s) A_n(s) U\( \frac{n}{4s} \) \( \pa_s - \frac{i n}{2s^2} \D \) D\(\frac{n}{2}\) \m{K}\f_n(s) ds 
% 	\\={}&: I_{1,n} +I_{2,n} +I_{3,n}.
\end{aligned}
	\label{im1:1}
\end{align}
Combining \eqref{eq:target}, \eqref{eq:factorize}, and \eqref{im1:1}, we reach to
\begin{equation}\label{eq:decomposition}
\begin{aligned}
	&{}i\int_t^{\I}U(t-s) \m{P}(s) ds \\
	=&{} iU(t)\F^{-1} \sum_{n\neq0,1} i^{-\frac32(n-2)}g_n \int_t^{\I} E^{1-\frac{1}{n}}(s) U\( \frac{n}{4s} \) D\(\frac{n}{2}\) \m{K}\f_n(s) \frac{ds}{2s} \\
% 	&=-U(t)\F^{-1} \sum_{n\neq0,1} g_n i^{\frac32-1}E^{1-\frac{1}{n}}(t) A(t) U\( \frac{n}{4t} \) D\(\frac{n}{2}\) \m{K}\f_n(t)\\
% 	&\quad-U(t)\F^{-1} \sum_{n\neq0,1} g_n i^{\frac32-1}\int_t^{\I} E^{1-\frac{1}{n}}(s) s \pa_s \(s^{-1}A(s)\) U\( \frac{n}{4s} \) D\(\frac{n}{2}\) \m{K}\f_n(s) ds\\
% 	&\quad-U(t)\F^{-1} \sum_{n\neq0,1} g_n i^{\frac32-1}\int_t^{\I} E^{1-\frac{1}{n}}(s) A(s) U\( \frac{n}{4s} \) \( \pa_s - \frac{i n}{2s^2} \D \) D\(\frac{n}{2}\) \m{K}\f_n(s) ds\\
	=&{} -iD(t) \sum_{n\neq0,1} \frac{g_n}{2 i^{\frac32(n-1)}} E^n(t) D\(\frac{n}2\)^{-1}U\(-\frac{n}{4t}\)  \\ 
	&{}\qquad \qquad A_n(t) U\( \frac{n}{4t} \) D\(\frac{n}{2}\) \m{K}\f_n(t)\\
	&{}- i\int_t^\I U(t-s)D(s) \sum_{n\neq0,1} \frac{g_n}{2 i^{\frac32(n-1)}} E^n(s)D\(\frac{n}2\)^{-1}U\(-\frac{n}{4s}\)\\
	&{}\qquad \qquad s \pa_s \(s^{-1}A_n(s)\) U\( \frac{n}{4s} \) D\(\frac{n}{2}\) \m{K}\f_n(s) ds\\
	&{}- i\int_t^\I U(t-s)D(s) \sum_{n\neq0,1} \frac{g_n}{2 i^{\frac32(n-1)}} E^n(s)D\(\frac{n}2\)^{-1}U\(-\frac{n}{4s}\)\\
	&\qquad \qquad A_n(s) U\( \frac{n}{4s} \) \( \pa_s - \frac{i n}{2s^2} \D \) D\(\frac{n}{2}\) \m{K}\f_n(s) ds\\
	&=: I_1 + I_2 + I_3.
\end{aligned}
\end{equation}
% In view of \eqref{eq:target}, we shall estimate $I_{j,n}$ ($j=1,2,3$) in $L^2$.
It will turn out that the term $I_1$ contains the main part and that $I_2$ and $I_3$ are remainder terms.

\subsection{Estimate of reminders}
Let us estimate $I_2$ and $I_3$ defined in in \eqref{eq:decomposition}.
% We first estimate $I_2$.
The following estimate is crucial.
\begin{lemma} \label{useful:1:1}
Let $3/2 < \d < 5/3$ and $\eta > \frac12 \(  \d -\frac{3}2\)$. 
Let $\psi(x) \in \m{S}$ and set
$\m{K}(t,n):=\m{K}_{\psi}(t,n)$ as in \eqref{def:Kpsi}.
Then, it holds for any $t\ge1$ and $n\neq 0,1$ that
\begin{multline}\label{eq:useful}
	\Lebn{A_n(t) U\(\frac{n}{4t} \) D\(\frac{n}{2}\) \m{K}\f_n(t)}{2} \\
	\le C t^{-\frac{\d}{2}}|n|^{-\d+\eta} \( \Sobn{\phi_n(t)}{\d} + \Lebn{|\xi|^{-\d}\f_n(t)}{2} \).
\end{multline}
\end{lemma}
% \begin{remark}
% When $d=2$, we need $\delta < 3/2$ to choose $\eta > 1=d/2$ in the above lemma.
% \end{remark}
Lemma \ref{useful:1:1} is proved in \cite{MM2} if $d=1,2$.
Although the proof for $d=3$ is essentially the same, 
we give it for self-containedness. 
\begin{proof}[Proof of Lemma \ref{useful:1:1}] 
We set $B(t) = (1+t|x|^2)^{-\frac{1}{2}}$, which yields $|A_n(t)| \le CB(t)^2$ for any $n\neq0,1$.
Then we have $|x|^{\t}B(t)^2  \le Ct^{-\frac{\t}{2}}$ for any $\t \in [0,2]$ 
and $B^{2} \in L^{(3/2)+\e} \cap L^\I(\R^d)$ for all $\e >0$.

By the triangle inequality, 
\begin{align*}
	\Lebn{B(t)^{2} U\(\frac{n}{4t} \) D\(\frac{n}{2}\) \m{K}\f_n(t)}{2} 
	\le{} &\Lebn{B(t)^{2}\( U\(\frac{n}{4t} \) -1 \) D\(\frac{n}{2}\) \m{K}\f_n(t)}{2} \\
	&+ \Lebn{B(t)^{2} D\(\frac{n}{2}\) \( \m{K}-1\) \f_n(t)}{2} \\
	&+ \Lebn{B(t)^{2} D\(\frac{n}{2}\) \f_n(t)}{2}\\
	=:{} & \mathrm{I}_n + \mathrm{II}_n+ \mathrm{III}_n
\end{align*}
For any $p_1>2$, one sees from Sobolev embedding and Lemma \ref{mol:1.1} (i) that
\[
\begin{aligned}
	\mathrm{I}_n
	&\le C\Lebn{B(t)^{2}}{p_1} \Lebn{|\nabla|^{\frac{3}{p_1}} \left| \frac{n|\n|^2}{t}\right|^{\frac12(\d-\frac{3}{p_1})} D\(\frac{n}{2}\) \m{K}\f_n(t)}{2} \\
	&\le Ct^{-\frac{\d}{2}} |n|^{-\d + (\frac{\d}2-\frac{3}{2p_1})} \norm{\f_n(t)}_{\dot{H}^{\d}}. 
\end{aligned}
\]
By definition of $\eta$, we are able to choose $p_1$ so that
\begin{equation*}
	\frac{\d}2-\frac{3}{2p_1} < \eta.
\end{equation*}
By Lemma \ref{mol:1.1} (ii), we estimate
\begin{align*}
	\mathrm{II}_n
	&\le C\Lebn{B^{2}}{p_2} \Lebn{|\n|^{\frac{3}{p_2}} D\(\frac{n}{2}\) (\m{K}-1)\f_n(t)}{2} \\
	&\le Ct^{-\frac{3}{2p_2}} |n|^{-\frac{3}{p_2}} \Lebn{|\n|^{\frac{3}{p_2}}\(\m{K}-1\)\f_n(t)}{2} \\
	&\le Ct^{-\frac12(\frac{3}{p_2}+ \t_2)} |n|^{-\frac{3}{p_2}- \t_2} \norm{\f_n(t)}_{\dot{H}^{\frac{3}{p_2}+\t_2}}
\end{align*}
for any $p_2\in (2,\I]$ and $\theta_2\in [0,2]$.
Taking $p_2$ and $\theta_2$ so that $\t_2 + \frac{3}{p_2} =\d$, we obtain desired estimate for $\mathrm{II}$.
Finally, we have
\begin{align*}
	&\mathrm{III}_n
	\le Ct^{-\frac{\d}{2}} \Lebn{|\xi|^{-\d}D\(\frac{n}{2}\) \f_n(t)}{2} \le Ct^{-\frac{\d}{2}} |n|^{-\d} \Lebn{|\xi|^{-\d}\f_n(t)}{2}.
\end{align*}
These estimates yield
\begin{multline*}
\Lebn{B^{2} U\(\frac{n}{4t} \) D\(\frac{n}{2}\) \m{K}\f_n(t)}{2} \\
\le Ct^{-\frac{\d}{2}}|n|^{-\d+\eta} \( \norm{\f_n(t)}_{H^{\d}} + \Lebn{|\xi|^{-\d}\f_n(t)}{2} \).
% \label{est25:1}
\end{multline*}
This completes the proof.
\end{proof}

Let us now give the estimate on $I_2$ and $I_3$.
\begin{lemma}
There exists $C=C(g_1,\norm{u_+}_{H^{0, 2}\cap H^{-\delta}})>0$ such that
\[
	\norm{I_2+I_3}_{L^\I([T,\I);L^2) \cap L^2([T,\I);L^6)}
	\le CT^{-\frac{\d}{2}} (\log T)^3\sum_{n \ne 0,1} |n|^{1+\eta} |g_n|
\]
for any $T\ge2$.
\end{lemma}
% Let us continue the proof of Proposition \ref{non:im1}.
\begin{proof}
By Strichartz' estimate, the identity $\pa_s \(s^{-1}A(s)\) = -2s^{-2} A(s) + s^{-2} \(A(s)\)^2$, and Lemma \ref{useful:1:1},
we compute
\begin{equation}\label{main:1:2}
\begin{aligned}
	&\norm{I_2}_{L^\I(T,\I;L^2) \cap L^2(T,\I;L^6)}\\
	&\le C
	\sum_{n\neq 0,1}|g_n| \int_T^\I \norm{A(s) U\( \frac{n}{4s} \) D\(\frac{n}{2}\) \m{K}\f_n(s)}_{L^2} \frac{ds}s\\
% 	&\quad + C\sum_{n\neq 0,1}|g_n| \int_T^\I \norm{A(s)^2 U\( \frac{n}{4s} \) D\(\frac{n}{2}\) \m{K}\f_n(s)}_{L^2} \frac{ds}{s}\\
	&\le C\sum_{n\neq 0,1} |g_n|
	|n|^{-\d+\eta} \int_T^{\I} s^{-\frac{\d}{2}} \norm{\phi_n(s)}_{\dot{H}^{\d} \cap H^{0, -\d}} \frac{ds}s.
\end{aligned}
\end{equation}
% Using Lemma \ref{useful:1:1}, we obtain
% \begin{align}
% \begin{aligned}
% 	\Lebn{I_1}{2}&=\Lebn{E^{1-\frac{1}{n}}(t) A(t) U\(\frac{n}{4t} \) D\(\frac{n}{2}\) \m{K}\f_n(t)}{2} \\
% 	&\le C t^{-\frac{\d}{2}}|n|^{-\d+\eta} \norm{\phi_n(t)}_{\dot{H}^{\d} \cap H^{0, -\d}}.
% \end{aligned}
% \label{main:1:1}
% \end{align}

We estimate $\Lebn{I_3}{2}$. We introduce the regularizing operators
$\m{K}_j := \m{K}_{\psi_j}$ ($j=1,2$) by \eqref{def:Kpsi} with
\begin{align*}
	&\psi_{1}(x) = -\frac{1}2 x\cdot \nabla \psi_0 \in \Sch, &
	&\psi_{2}(x) = \frac{i}2  |x|^2 \psi_0(x) \in \Sch.
\end{align*}
Remark that $\nabla \psi_1(0)=\nabla \psi_2(0)=0$.
We then have an identity
\begin{align*}
	\(\pa_s - \frac{in}{2s^2} \D \) D\(\frac{n}{2}\) \m{K}\f_n ={}&  D\(\frac{n}{2}\) \m{K}\pa_s\f_n + s^{-1} D\(\frac{n}{2}\) \m{K}_1 \f_n \\
	&{}+ s^{-1} n  D\(\frac{n}{2}\) \m{K}_2 \f_n.
\end{align*}
Since $\m{K}_1$ and $\m{K}_2$ of the form \eqref{def:Kpsi}, 
the estimate \eqref{eq:useful} is valid also for these regularizing operators.
Then, mimicking the estimate of $I_2$, we have
\begin{align}
\begin{aligned}
	&\norm{I_3}_{L^\I(T,\I;L^2) \cap L^2(T,\I;L^6)}\\
	&\le C\sum_{n\neq0,1}|g_n| |n|^{-\d+\eta}  \int_T^{\I} s^{-\frac{\d}{2}} \norm{\pa_s \phi_n(s)}_{\dot{H}^{\d} \cap H^{0, -\d}} ds\\
	&\qquad + C \sum_{n\neq0,1}|g_n| |n|^{-\d+\eta} \int_T^{\I} s^{-\frac{\d}{2}-1}  \norm{\phi_n(s)}_{\dot{H}^{\d} \cap H^{0, -\d}} ds \\
	&\qquad + C \sum_{n\neq0,1}|g_n| |n|^{-\d+1+\eta} \int_T^{\I} s^{-\frac{\d}{2}-1}  \norm{\phi_n(s)}_{\dot{H}^{\d} \cap H^{0, -\d}} ds
\end{aligned}
\label{main:1:3}
\end{align}
for $T\ge 2$.
By %\eqref{im1:1}, \eqref{main:1:1}, 
\eqref{main:1:2}, \eqref{main:1:3}, Lemmas \ref{lem1:2a} and \ref{lem1:2}, and the estimates
\begin{align*}
	\norm{\phi_n}_{H^{0,-\delta}}
	&\le C\norm{\wha{u_+}}_{L^\I}^{\frac23} \norm{u_+}_{\dot{H}^{-\delta}}, \\
	\norm{\partial_t \phi_n}_{H^{0,-\delta}}
	&\le C\frac{|g_1|}{t}\norm{\wha{u_+}}_{L^\I}^{\frac43} \norm{u_+}_{\dot{H}^{-\delta}},
\end{align*}
we obtain the desired estimate.
\end{proof}

\subsection{Estimates on the main contribution} 
We estimate $I_1$ in \eqref{eq:decomposition}.
Recall that
\[
	\mathcal{V}
	= - \F^{-1} \sum_{n\neq0,1} \frac{g_n}{2i^{\frac32 n}} M\(-\frac{n}{4t}\)   A_n(t)  D\(\frac{n}{2}\) \f_n(t).
\]
With the following proposition, we obtain \eqref{eq:nrV}.
\begin{proposition}
% The main part of $I_1$ is $\mathcal{V}$.
There exists $C=C(g_1,\norm{u_+}_{H^{0, \frac53}})>0$ such that
\begin{equation*}%\label{eq:I1V}
	\norm{I_1 - \mathcal{V}}_{
	L^\I ([T,\I);L^2) \cap L^2 ([T,\I);L^6)} \le C T^{-\frac\d2}\Jbr{\log T}^3 \sum_{n\neq0,1} |n|^{1+\eta}|g_n|
\end{equation*}
holds for any $T\ge2$.
% Moreover, the second contribution is small in $L^\I (T,\I;L^2)$ in such a sense that 
% \begin{equation}\label{eq:Vest}
% 	\norm{\mathcal{V}}_{L^\I (T,\I;L^2)} \le C \norm{u_+}_{H^{-\delta} \cap H^{0, 2}}^{\frac53} T^{-\frac\d2} \sum_{n\neq0,1} \Jbr{n}^{-\d}|g_n|
% \end{equation}
% for $T\ge2$.
% Furthermore, the second contribution is approximated by $v$ in $L^2 (T,\I;L^6)$:
% There exists $C=C(g_1,\norm{u_+}_{H^{-\delta} \cap H^{0, 2}})>0$ such that
% \begin{equation}\label{eq:Vest2}
% 	\norm{\mathcal{V}-v_p}_{L^2 (T,\I;L^6)} \le 
% 	C T^{-\frac\d2}\Jbr{\log T}^3 \sum_{n\neq0,1} \Jbr{n}^\frac56|g_n|.
% \end{equation}
% for $T\ge2$.
\end{proposition}
\begin{proof}
We further break up $I_1$ as
\begin{align*}
	I_1 ={}& 
	-iD(t) \sum_{n\neq0,1} \frac{g_n}{2i^{\frac32(n-1)}} E^n(t) D\(\frac{n}2\)^{-1}U\(-\frac{n}{4t}\)  A_n(t) \(U\( \frac{n}{4t} \)-1\) D\(\frac{n}{2}\) \m{K}\f_n(t)\\
	&-iD(t) \sum_{n\neq0,1} \frac{g_n}{2i^{\frac32(n-1)}} E^n(t) D\(\frac{n}2\)^{-1}U\(-\frac{n}{4t}\)  A_n(t) D\(\frac{n}{2}\) (\m{K}-1)\f_n(t)\\
	&-iD(t) \sum_{n\neq0,1} \frac{g_n}{2i^{\frac32(n-1)}} E^n(t) D\(\frac{n}2\)^{-1}U\(-\frac{n}{4t}\)  A_n(t) D\(\frac{n}{2}\) \f_n(t)\\
	&=: \mathrm{IV} + \mathrm{V} + \mathrm{VI}.
\end{align*}
A computation shows that $\mathrm{VI}=\mathcal{V}$. Since $|A_n(t)|\le 1$, we have
\begin{align*}
	\norm{\mathrm{IV}}_{L^2_x} 
	&\le C \sum_{n\neq0,1} |g_n| \norm{ \(\frac{|n|}{t}\)^{\d/2} |\n|^\d D\(\frac{n}{2}\) \m{K}\f_n(t) }_{L^2}\\
	&\le C t^{-\frac\d2} \sum_{n\neq0,1} |n|^{-\frac\d2} |g_n| \norm{\phi_n(t)}_{\dot{H}^\d}
\end{align*}
and
\begin{align*}
	\norm{\mathrm{V}}_{L^2_x} 
	&\le C \sum_{n\neq0,1} |g_n| \norm{(\m{K}-1)\f_n(t) }_{L^2}
	\le C t^{-\frac\d2} \sum_{n\neq0,1} |n|^{-\d} |g_n| \norm{\phi_n(t)}_{\dot{H}^\d}.
\end{align*}
Hence, we have $L^\I(T,\I;L^2)$-estimate.
Similarly, by $L^p-L^q$ estimate of the Schr\"odinger group, the H\"older estimate,
Sobolev embedding, and Lemma \ref{mol:1.1} (i), 
we have
\begin{align*}
	\norm{\mathrm{IV}}_{L^6_x}
	&\le C \sum_{n\neq0,1} |g_n| \norm{B(t)^2 \(U\( \frac{n}{4t} \)-1\) D\(\frac{n}{2}\) \m{K}\f_n(t) }_{L^{\frac65}}\\
	&\le C \sum_{n\neq0,1} |g_n| \Lebn{B^{2}(t)}{p_4} \Lebn{|\nabla|^{\frac{3}{p_4}-1} \left| \frac{n|\n|^2}{t}\right|^{\frac\d{2}+\frac12-\frac{3}{2p_4}} D\(\frac{n}{2}\) \m{K}\f_n(t)}{2} \\
	&\le Ct^{-\frac{\d}{2}-\frac12} \sum_{n\neq0,1} |n|^{-\d + (\frac{\d}2+\frac12-\frac{3}{2p_4})} |g_n| \norm{\f_n(t)}_{\dot{H}^{\d}}
\end{align*}
for any $3 \ge p_4>3/2$.
By definition of $\eta$, we are able to choose $p_4$ so that
\begin{equation*}
	\frac{\d}2+ \frac12-\frac{3}{2p_4} < 1+\eta.
\end{equation*}
By H\"older's inequality and Lemma \ref{mol:1.1} (ii), we obtain
\begin{align*}
	\norm{\mathrm{V}}_{L^6_x}
	&\le C \sum_{n\neq0,1} |g_n| \Lebn{B(t)^{2}}{3} \Lebn{(\m{K}-1)\f_n(t)}{2} \\
% 	&\le Ct^{-\frac{1}{2}} \Lebn{\(\m{K}-1\)\f_n(t)}{2} \\
	&\le Ct^{-\frac{\d}2-\frac12} \sum_{n\neq0,1} |n|^{-\d}|g_n| \norm{\f_n(t)}_{\dot{H}^{\d}}.
\end{align*}
This competes the proof.
\end{proof}

% \subsection{Completion of the proof}
We are in a position to finish the proof of Proposition \ref{non:im1}.

\begin{proof}[Proof of Proposition \ref{non:im1}]
It suffices to establish \eqref{eq:Vest} and \eqref{eq:Vest2}.
The estimate \eqref{eq:Vest} follows from
\[
	\norm{\mathcal{V}}_{L^2} \le C \sum_{n\neq0,1} |g_n| \norm{B(t)^2 D\(\frac{n}2\) \f_n(t)}_{L^2}.
\]
The right hand side is $\mathrm{III}_n$ in the proof of Lemma \ref{useful:1:1}.

Finally, we prove \eqref{eq:Vest2}. Since
\begin{align*}
\mathcal{V}
	&=v_p - \sum_{n\neq0,1} \frac{g_n}{2i^{\frac32(n+1)}} C_n(t) M\(\frac{t}{n}\) D (t)\( U\(-\frac{1}{4nt}\)-1\)  \f_n(t),
\end{align*}
where $C_n(t):=\F^{-1} A_n(t) \F=(1+i(\frac{n-1}{n})t\Delta)^{-1}$.
Since $\norm{\nabla C_n(t)}_{\mathcal{L}(L^2)} \le Ct^{-1/2}$ for any $n\neq 0,1$ and $t\ge2$,
we see from Sobolev embedding that
\begin{align*}
	\norm{\mathcal{V}-v_p}_{L^6_x}
	\le& C t^{-\frac12}\sum_{n\neq0,1} |g_n| \norm{\( U\(-\frac{1}{4nt}\)-1\)  \f_n(t)}_{L^2}\\
	\le& C t^{-\frac\d2-\frac12}\sum_{n\neq0,1} |n|^{-\frac\d2}|g_n| \norm{\f_n(t)}_{\dot{H}^\d}.
\end{align*}
Hence, we have the desired estimate.
\end{proof}

We finally give an outline to obtain the asymptotics of $\mathcal{V}(t)$ in Remark \ref{rmk:MTT}.
Note that
\[
	\mathcal{V}
	= - \F^{-1} \sum_{n\neq0,1} \frac{g_n}{2i^{\frac32 n}} M\(-\frac{n}{4t}\)  |x|^2 A_n(t)  D\(\frac{n}{2}\) (n^{-2}|x|^{-2}) \f_n(t).
\]
As $|x|^2 A_n(t)=\frac{n}{i(n-1)t}(1-A_n(t))$,
\begin{align*}
	\mathcal{V}
	={}& - \F^{-1} \sum_{n\neq0,1} \frac{ g_n}{2i^{1+\frac32 n}n(n-1)t} M\(-\frac{n}{4t}\) D\(\frac{n}{2}\) |x|^{-2} \f_n(t)\\
	&{}+ \F^{-1} \sum_{n\neq0,1} \frac{ g_n}{2i^{1+\frac32 n}n(n-1)t} M\(-\frac{n}{4t}\) A_n(t) D\(\frac{n}{2}\) |x|^{-2} \f_n(t)\\
	=:{}&\mathrm{VII}+\mathrm{VIII}.
\end{align*}
By $\norm{|x|^\zeta A_n(t)}_{\mathcal{L}(L^2)} \lesssim t^{-\frac\zeta2}$ for any $\zeta \in [0,2]$,
$\mathrm{VIII}$ is small if $\wha{u_+} \in L^{\I} \cap H^{0, -2-}$.
% we have 
% $$\norm{\mathrm{VIII}}_{L^2} \lesssim t^{-1-\frac\zeta2}\norm{|x|^{-2-\zeta} |\wha{u_+}|^{5/3}}_{L^2}\sum_{n\neq0,1}
% |n|^{-\frac12-\zeta}|g_n|.$$
Further, since $\F^{-1}M\(-\frac{n}{4t}\)=U(\frac{t}n)\F^{-1}=M(\frac{t}n)D(\frac{t}n)U(-\frac{n}{4t}) \sim M(\frac{t}n)D(\frac{t}n)$,
\begin{align*}
	\mathrm{VII} =&
	-\sum_{n\neq0,1} \frac{ g_n}{2i^{1+\frac32 (n+1)}n(n-1)t} M\(\frac{t}n\) D\(t\) |x|^{-2} \f_n(t)
	+o(t^{-1})\\
	=& \sum_{n\neq0,1} \frac{g_n}{n(1-n)}  \abs{\frac{2t}{x}}^2 |u_p(t)|^{\frac53-n} u_p(t)^n
	+o(t^{-1})
\end{align*}
as $t\to\I$ for suitable $u_+$. We omit the detail.

\appendix
\section{A calculation of Fourier coefficients}\label{sec:appendix1}

In this appendix, we demonstrate an explicit formula of Fourier coefficients of the function $g(\theta) = 
|\cos \theta |^{\alpha-1}\cos \theta $.
This contains our example in Remark \ref{rmk:example}.

\begin{proposition}
Let $\alpha>-1$ be not an odd integer.
Let 
\[
	g_n := \frac1{2\pi} \int_{-\pi}^\pi |\cos \theta |^{\alpha-1}\cos \theta \cos n\theta d\theta
\]
 for $n\in \Z$.
Then, $g_n=0$ for even $n$ and 
\begin{equation}\label{eq:explicit}
	g_n = \frac{(-1)^{\frac{n-1}2} \Gamma(\frac{\alpha+2}2) \Gamma(\frac{n-\alpha}2)}{\sqrt\pi 
	\Gamma(- \frac{\alpha-1}2) \Gamma(\frac{n+\alpha+2}2)}
\end{equation}
for odd $n$.
In particular, $g_n = O(|n|^{-\alpha-1})$ as $|n|\to\I$.
\end{proposition}
\begin{proof}
$g_n=0$ for even $n$ is obvious.
For odd $n$, by the symmetry we have
\[
	g_n = \frac1{\pi} \int_{-\frac\pi2}^{\frac{\pi}2} \cos^{\alpha} \theta \cos n\theta d\theta
\]
Let $a_m := {g}_{2m+1}$ for $m\in \Z$.
We first show that there exists a constant $c_\alpha\in \R$ such that 
\begin{equation}\label{eq:app_1}
	a_m = c_\alpha (-1)^m \frac{\Gamma(m-\frac{\alpha-1}{2})}{\Gamma(m+\frac{\alpha+3}{2})}
\end{equation}
for $m\in \Z$. By integration by parts,
% \begin{align*}
% 	a_m - a_{m-1} ={}&
% 	 \frac2{\pi} \int_{-\frac\pi2}^{\frac{\pi}2} \cos^{\alpha} \theta \sin 2m\theta ( \cos \theta)' d\theta\\
% 	 ={}& \frac{2\alpha}{\pi} \int_{-\frac\pi2}^{\frac{\pi}2} \cos^{\alpha} \theta \sin 2m\theta  \sin \theta d\theta\\
% 	 &{} - \frac{4m}{\pi} \int_{-\frac\pi2}^{\frac{\pi}2} \cos^{\alpha} \theta \cos 2m\theta  \cos \theta d\theta\\
% 	 ={}& -\alpha (\alpha_m-\alpha_{m-1}) - 2m (a_m + a_{m-1}).
% \end{align*}
\begin{align*}
	a_m - a_{m-1} ={}&
	 \frac2{\pi(\alpha+1)} \int_{-\frac\pi2}^{\frac{\pi}2} \sin 2m\theta ( \cos^{\alpha+1} \theta)' d\theta\\
	 ={}& - \frac{4m}{\pi(\alpha+1)} \int_{-\frac\pi2}^{\frac{\pi}2} \cos^{\alpha} \theta \cos 2m\theta  \cos \theta d\theta\\
	 ={}&  - \frac{2m}{(\alpha+1)} (a_m + a_{m-1}).
\end{align*}
Hence, we obtain the recurrence relation $a_m = - \frac{m-\frac{\alpha+1}2}{m+\frac{\alpha+1}2} a_{m-1}$.
This shows \eqref{eq:app_1} because the right hand side satisfies the same relation.
Further, since
\[
	a_0 = \frac1{\pi} \int_{-\frac\pi2}^{\frac{\pi}2} \cos^{\alpha+1} \theta  d\theta = \frac{\Gamma(\frac{\alpha+2}2)}{\sqrt\pi \Gamma (\frac{\alpha+3}2)},
\]
we have $c_\alpha = \Gamma(\frac{\alpha+2}2)/ \sqrt{\pi} \Gamma(\frac{1-\alpha}2)$, which shows \eqref{eq:explicit}
together with \eqref{eq:app_1}.
The last assertion easily follows by means of the Stirling formula.
\end{proof}

A similar argument shows the following

\begin{proposition}
Let $\alpha>-1$ be not an odd integer.
Let 
\[
	g_n := \frac1{2\pi} \int_{-\pi}^\pi |\sin \theta |^{\alpha-1}\sin \theta \sin n\theta d\theta
\]
 for $n\in \Z$.
Then, $g_n=0$ for even $n$ and 
\begin{equation}\label{eq:explicit}
	g_n = \frac{\Gamma(\frac{\alpha+2}2) \Gamma(\frac{n-\alpha}2)}{\sqrt\pi 
	\Gamma(- \frac{\alpha-1}2) \Gamma(\frac{n+\alpha+2}2)}
\end{equation}
for odd $n$.
In particular, $g_n = O(|n|^{-\alpha-1})$ as $|n|\to\I$.
\end{proposition}
\begin{proof}
$g_n=0$ for even $n$ is obvious.
For odd $n$, by the symmetry we have
\[
	g_n = \frac1{\pi} \int_{0}^{\pi} \sin^{\alpha} \theta \sin n\theta d\theta
\]
Let $b_m:=g_{2m+1}$ for $m\in \Z$.
We have the recurrence relation $$b_m =  \frac{m-\frac{\alpha+1}2}{m+\frac{\alpha+1}2} b_{m-1}$$
since
\begin{align*}
	b_m + b_{m-1} ={}&
	 \frac2{\pi(\alpha+1)} \int_{0}^{\pi} \cos 2m\theta ( \sin^{\alpha+1} \theta)' d\theta\\
	 ={}& - \frac{4m}{\pi(\alpha+1)} \int_{0}^{\pi} \sin^{\alpha} \theta \sin 2m\theta  \sin \theta d\theta\\
	 ={}&  \frac{2m}{(\alpha+1)} (-b_m + b_{m-1}).
\end{align*}
Together with 
$b_0 = \frac{\Gamma(\frac{\alpha+2}2)}{\sqrt\pi \Gamma (\frac{\alpha+3}2)}$,
we obtain the result as in the previous case.
\end{proof}

\vskip3mm
\noindent {\bf Acknowledgments.} 
S.M. is partially supported by Sumitomo Foundation, Basic Science Research
Projects No.\ 161145 and by JSPS, Grant-in-Aid for Young Scientists (B) 17K14219.

% \bibliographystyle{amsplain}
% \bibliography{homogeneous_0521}


\end{document}